\documentclass{amsart}

\usepackage{epsfig,amsmath}
\usepackage{xspace}
\usepackage[psamsfonts]{amssymb}
\usepackage[latin1]{inputenc}
\usepackage{graphicx,color}
\usepackage[curve]{xypic}

\usepackage{hyperref}

\theoremstyle{remark}

\newtheorem{example}{\bf Example}

\newtheorem{remark}{\bf Remark}

\theoremstyle{plain}

\newtheorem{proposition}{\bf Proposition}
\newtheorem{definition}{\bf Definition}
\newtheorem{theorem}{\bf Theorem}
\newtheorem*{theorem*}{\bf Theorem}
\newtheorem{lemma}{\bf Lemma}
\newtheorem{corollary}{\bf Corollary}

\def\C{{\mathbb C}}
\def\D{{\mathbb D}}

\def\R{{\mathbb R}}
\def\P{{\mathbb P}}

\def\J{{\mathcal J}}
\def\K{{\mathcal K}}
\def\I{{\mathcal I}}
\def\Part{{\mathrm{Part}(I)}}
\def\part{{\mathrm{Part}^*(I)}}

\def\cal{\mathcal}

\def\eqdef{{:=}}

\def\id{{\rm id}}

\def\ds{\displaystyle}

\def\rad{{{\vartheta}_{\rm rad}}}
\def\theta{{\vartheta}}

\def\ie{{\it{i.e.}},}

\title{B\"ottcher coordinates}
\subjclass{}
\begin{author}[X.~Buff]{Xavier Buff}\thanks{The research of the first author was supported in part by the grant ANR-08-JCJC-0002 and the IUF}
\email{xavier.buff@math.univ-toulouse.fr}
\address{ %
  Institut de Math\'ematiques de Toulouse\\
 Universit\'e Paul Sabatier\\
  118, route de Narbonne \\
  31062 Toulouse Cedex \\
  France }
\end{author}
\begin{author}[A.L.~Epstein]{Adam L. Epstein}
\email{adame@maths.warwick.ac.uk}
\address{ %
Mathematics institute\\
 University of Warwick\\
 Coventry CV4 7AL\\
 United Kingdom }
\end{author}
\begin{author}[S.~Koch]{Sarah Koch}\thanks{The research of the third author was supported in part by the NSF}
\email{kochs@math.harvard.edu}
\address{ %
Department of Mathematics\\
Harvard University\\
Cambridge MA 02138\\
United States
}
\end{author}

\begin{document}

\maketitle

\begin{abstract}
A well-known theorem of B\"ottcher asserts that an analytic germ
$f:(\C,0)\to (\C,0)$ which has a superattracting fixed point at $0$,
more precisely of the form $f(z) = az^k + o(z^k)$ for some $a\in
\C^*$, is analytically conjugate to $z\mapsto az^k$ by an analytic
germ $\phi:(\C,0)\to (\C,0)$ which is tangent to the identity at
$0$. In this article, we generalize this result to analytic maps of
several complex variables.
\end{abstract}

\section*{Introduction}

An analytic germ $f:(\C,0)\to(\C,0)$ with a superattracting fixed
point at $0$ can be written in the form
\[
f(z)=a z^k+{\cal O}(z^{k+1}),\qquad a\neq 0,\;\;k\geq 2.
\]
In 1904, L. E. B\"ottcher proved the following theorem.

\begin{theorem*}[B\"ottcher]\label{bottcherone}
If $f:(\C,0)\to(\C,0)$ has a superattracting point at $0$ as above,
there exists a germ of analytic map $\phi:(\C,0)\to(\C,0)$, which is
tangent to the identity at $0$ and conjugates $f$ to the map
$h:w\mapsto a w^k$ in some neighborhood of $0$, \ie{} $\phi\circ f =
h\circ \phi$.
\end{theorem*}

The germ $\phi$ is called a {\em{B\"ottcher coordinate}} for the
germ $f$.

B\"ottcher coordinates have been an essential tool in the study of
complex dynamics in one variable; a B\"ottcher coordinate gives
polar coordinates near the associated superattracting fixed point,
which are compatible with the dynamics of $f$. In the case of the
superattracting fixed point at infinity for a polynomial, the
angular coordinate is called the {\em{external angle}} and its level
curves are called {\em{external rays}}. The importance of external
angles and external rays was emphasized in \cite{dh}.

\begin{figure}[htbp] 
   \centering
   \includegraphics[width=5in]{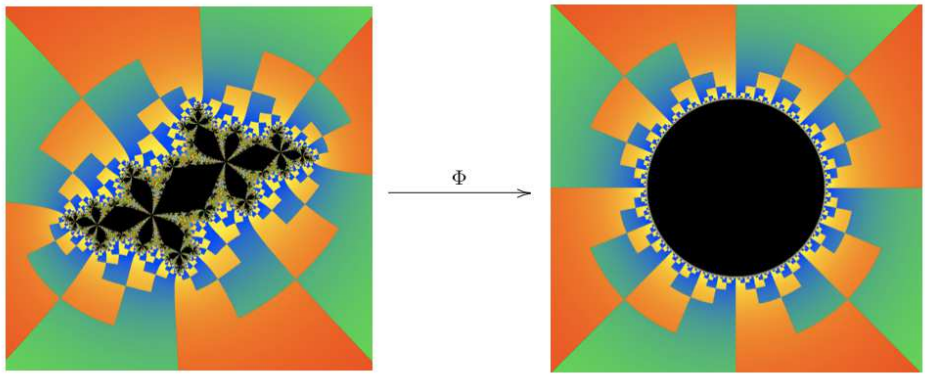}
   \caption{For polynomials of one complex variable, the B\"ottcher coordinate of $\infty$
   extends throughout the entire basin if the filled Julia set is connected. The checkerboard
   pattern in the pictures above highlights the external rays in each picture.
   On the right is the filled Julia set of the model map, $z\mapsto z^2$, and on the left
   is the filled Julia set for a quadratic polynomial.}
   \label{polarcoords}
\end{figure}

It is quite natural to ask whether there is an analogue of this
theorem in higher dimensions; that is, given an analytic germ
$F:(\C^m,0)\to (\C^m,0)$ is there an analytic germ $\Phi:(\C^m,0)\to
(\C^m,0)$ which conjugates $F$ to its terms of lowest degree?
According to Hubbard and Papadopol in \cite{hp}, ``the map is not in
general locally conjugate, even topologically, to its terms of
lowest degree; the local geometry near such a point is much too rich
for anything like that to be true.''  Hubbard and Papadopol present
the following example to illustrate their point.

\begin{example}\label{firstexample}
Consider the map
\[
F:\C^2\to \C^2\text{ given by }(x,y)\mapsto (x^2+y^3,y^2).
\]
Let $H:\C^2\to\C^2$ be the map $(x,y)\mapsto (x^2,y^2)$. There is no
analytic conjugacy between $F$ and $H$ in a neighborhood of the
origin because the dynamics of the maps are incompatible. Indeed,
the critical locus of $F$ consists of the the two axes, which is
also the critical locus of $H$. But the map $H$ fixes the two axes,
whereas $F$ fixes $y=0$ and maps $x=0$ to the curve $t\mapsto
(t^3,t^2)$.

However, there is another explanation which is more relevant to our
discussion. The critical value locus of $F$ contains the curve
$t\mapsto (t^3,t^2)$ which has a cusp; any smooth conjugacy would
have to map this singular curve to a component of the critical value
locus of $H$; however, the critical value set of $H$ consists of
$x=0$ and $y=0$, which are smooth. So certainly no analytic
conjugacy exists.
\end{example}

\begin{remark}\label{firstremark}
In the example above, the map $H:\C^{m+1}\to\C^{m+1}$ is {\em{homogeneous}} and {\em{nondegenerate}}. It descends to an endomorphism of projective space $h:{\mathbb
P}^m\rightarrow{\mathbb P}^m$; the critical locus and the
postcritical locus of $H$ will be cones over the critical locus and
postcritical locus of $h$.

Therefore, if $F$ is going to be locally conjugate to a homogeneous
map in a neighborhood of $0$, each component  of the
critical locus and postcritical locus of $F$ containing the
superattracting fixed point, must be an {\em{analytic cone}}; that
is, each of these components must be the image of a cone under an
analytic isomorphism. This is a rather strong condition; the
homogeneous maps certainly satisfy it, and we will see in section
\ref{applications} that there are other families of maps which
satisfy this criterion.
\end{remark}

The following people have studied the dynamics of maps with
superattracting fixed points in higher dimensions: J.H. Hubbard and
P. Papadopol develop a very important theory of Green functions for
maps with superattracting fixed points in $\C^m$ (see \cite{hp}), C.
Favre and M. Jonsson classify {\em{contracting rigid germs}} of
$(\C^2,0)$ in \cite{fj}. Both S. Ushiki and T. Ueda have results
about a B\"ottcher theorem in higher dimensions: Ushiki presents a
local B\"ottcher theorem in $\C^2$ for maps of a special form in
\cite{us}, and Ueda presents both local and global B\"ottcher
theorems in $\C^m$ for a particular family of maps in \cite{ue}. 

In this article, we begin with an analytic germ $F:(\C^m,0)\to
(\C^m,0)$, which has an {\em{adapted superattracting fixed point}}
at $0$. We present necessary and sufficient conditions for such an
$F$ to be locally conjugate to its {\em{quasihomogeneous part}} at
$0$. Our criteria are stated in terms of {\em{admissible vector
fields}}; these vector fields detect the shape of the postcritical
locus of our map $F$ to see if the analytic cone condition in remark
\ref{firstremark} is satisfied. The relevant definitions  are
given in Section \ref{setup}.

\begin{theorem}[Local B\"ottcher Coordinates]\label{theo_local}
Let $F:(\C^m,0)\to (\C^m,0)$ be a germ of an analytic map having an
adapted superattracting fixed point at $0\in \C^m$. Let $H:\C^m\to
\C^m$ be the quasihomogeneous part of $F$ at $0$ with multidegree $(k_1,\ldots k_p)$. Then the following
are equivalent.
\begin{enumerate}
\item There is a germ of an analytic map $\Phi:(\C^m,0)\to(\C^m,0)$ such that
\[\Phi(0)=0, \quad D_0\Phi=\mathrm{id}\quad\text{and}\quad \Phi\circ F=H\circ
\Phi.\]
\item There is an admissible $p$-tuple of germs of vector fields $(\xi_1\ldots,\xi_p)$ such that
\[
DF\circ \xi_j=k_j\cdot \xi_j\circ F
\]
 near $0$ for all $j\in [1,p]$.
\item There is an admissible $p$-tuple of germs of vector fields $(\zeta_1,\ldots,\zeta_p)$ such that $\zeta_j$ is tangent
to the germ of the postcritical set of $F$ for all $j\in[1,p]$.
\end{enumerate}
\end{theorem}

The proof of theorem \ref{theo_local} is a beautiful confluence of
classical results of Euler, Poincar\'e, Cartan, and Arnol'd with the
theory of Green functions developed in \cite{hp}. Sections
\ref{qhmaps}--\ref{cartan} are devoted to presenting these results
and adapting them to our setting.

In section \ref{globaltheoremsection}, we present a result about
extending a local B\"ottcher coordinate $\Phi$ to a larger domain. More precisely let $\Omega$ be a complex analytic manifold, and $F:\Omega\to \Omega$ be an analytic map with a superattracting fixed point at $a\in\Omega$. The basin of attraction of $a$ for $F$ is the open set of points whose orbits converge to $a$. The immediate basin is the connected component of the basin containing $a$; we denote the immediate basin as ${\cal B}_a(F)$. 

\begin{theorem}[Global Böttcher coordinates]\label{theo_global}
Let $\Omega$ be a complex analytic manifold,  $F:\Omega\to \Omega$
be a proper analytic map with a superattracting fixed point at $a\in
\Omega$ and $H:T_a\Omega\to T_a\Omega$ be a quasihomogeneous map.
Suppose that
\begin{itemize}

\item there is a local isomorphism $\Phi:(\Omega,a)\to (T_a\Omega,0)$ with $\Phi\circ F = H\circ \Phi$;

\item near $a$, $\Phi$ maps the postcritical set of $F:{\cal B}_a(F)\to{\cal B}_a(F)$ to the postcritical set of $H:{\cal B}_0(H)\to {\cal B}_0(H)$.
\end{itemize}
Then, $\Phi$ extends to a global isomorphism $\Phi:{\cal B}_a(F)\to {\cal B}_0(H)$ conjugating $F$ to $H$:
\[\xymatrix{
&{\cal B}_a(F)\ar[rr]^{\Phi}\ar[d]^F & &{\cal B}_0(H)\ar[d]^H\\
&{\cal B}_a(F)\ar[rr]^{\Phi} & &{\cal B}_0(H).}
\]
\end{theorem}

In section \ref{applications}, we apply our results to a large class
of new examples of {\em{postcritically finite endomorphisms of
${\mathbb P}^n$}}. And finally, in section \ref{questions}, we conclude
with some remaining questions.

\subsection*{Acknowledgements} We would like to thank our colleagues for their constant support, and stimulating discussions. In particular, we thank John H. Hubbard, Curtis T. McMullen, John Milnor, and Saeed Zakeri. 

\section{The Local Result}\label{firstsection}

\subsection{Set Up}\label{setup}

Throughout section \ref{firstsection}, we will use the following notation. 
\begin{itemize}
\item $E$ is a $\C$-linear space of dimension $m\geq 1$.

\item $\langle\cdot|\cdot\rangle$ is a Hermitian product on $E$.

\item $\|\cdot\|$ is the associated norm.

\item $E=E_1\oplus \cdots \oplus E_p$ is the direct sum of $p\geq 1$ linear
spaces.

\item $(\pi_j:E\to E_j)_{j\in [1,p]}$ are the projections associated to the direct sum.

\item for $v\in E$, we denote $v_j\eqdef \pi_j (v)\in E_j$.

\item For $F:(E,0)\to (E,0)$ an analytic germ, we denote $F_j\eqdef \pi_j\circ F:E\to E_j$. 

\item if $(H_j:E_j\to E_j)_{j\in [1,p]}$ are maps, then
$H_1\oplus\cdots \oplus H_p$ is the map
\[E\ni v_1 + \cdots + v_p
\mapsto H_1(v_1) + \cdots + H_p(v_p)\in E.\]


\item $k_1,\ldots,k_p$ are integers greater than or equal to 2.

\end{itemize}
The word {\em adapted} in the definition \ref{def_adapted} refers to
this data.

Recall that a map $H:L\to L$ on a $\C$-linear space $L$ is
homogeneous of degree $k\geq 1$ if
\[\forall v\in L,\quad
H(v) = \phi(\underset{k\text{ times}}{\underbrace{v,\ldots,v}})\]
for some $k$-linear map $\phi:L^k\to L$. Equivalently, $H$ is
analytic and
\[\forall \lambda\in \C,~\forall v\in L,\quad
H(\lambda v) = \lambda^k H(v).\] The homogeneous map $H$ is
nondegenerate if $H^{-1}\{0\} = \{0\}$.

\begin{definition}\label{def_adapted}
An analytic germ $F:(E,0)\to (E,0)$ has an {\em adapted
superattracting fixed point} if for all $j\in [1,p]$, there is a
nondegenerate homogeneous map $H_j:E_j\to E_j$ of degree $k_j\geq 2$
such that
\[F_j(v) \underset{v\to 0}= H_j (v_j) + {\cal O}\bigl(\|v\|\cdot\|v_j\|^{k_j}\bigr).\]
The map $H_1\oplus \cdots \oplus H_p$ is the {\em
quasihomogeneous part} of $F$ at $0$ and $(k_1,\ldots,k_p)$ is its {\em multidegree}.
\end{definition}

Note that the derivative of a map at an adapted superattracting
fixed point necessarily vanishes. The particular form of $F_j$ implies the following invariance property. 
\begin{lemma}\label{prop_loctotinv}
The spaces $E_j^\top\eqdef\{v_j=0\}$ and $E_j$ are locally totally invariant under $F$; that is, there exists a neighborhood $V$ of $0$ such that 
\[
F^{-1}(E_j^\top)\cap V=E_j^\top\cap V\quad\text{and}\quad F^{-1}(E_j)\cap V=E_j\cap V.
\]
\end{lemma}
\begin{proof}
Let us first prove that $E_j^\top$ is locally totally invariant. First, if $v_j=0$, then $F_j(v)=0$. This shows that near $0$,  $E_j^\top\subseteq F^{-1}(E_j^\top)$. So if $E_j^\top$ were not locally totally invariant, we could find an analytic germ 
\[\gamma:(\C,0)\to (F^{-1}(E_j^\top),0)\quad\text{with}\quad \gamma_j\eqdef\pi_j\circ\gamma\not\equiv 0.\]
The order of vanishing $m$ of $\gamma_j$ at $0$ would be finite. Since $H_j$ is nondegenerate with degree $k_j$, the order of vanishing of $H_j\circ \gamma_j$ would be $k_j\cdot m$. So the order of vanishing of $F_j\circ \gamma=H_j\circ \gamma_j+o\bigl(\|\gamma_j\|^{k_j}\bigr)$ would be $k_j\cdot m<\infty$ which is a contradiction since by assumption $F_j\circ \gamma\equiv 0$. 

Next, note that
\[
E_j=\bigcap_{i\neq j} E_i^\top.
\]
Since the spaces $E_i^\top$ are locally totally invariant, it follows that the space $E_j$ is locally totally invariant. 
\end{proof}

We want to give necessary and sufficient conditions for the
existence of an isomorphism $\Phi:(E,0)\to (E,0)$ conjugating $F$ to
its quasihomogeneous part $H$, \ie{} such that $\Phi\circ F = H\circ
\Phi$. Such an isomorphism is called a Böttcher coordinate for $F$.

For  $v\in E$, there is a canonical isomorphism between the
$\C$-linear space $E$ and the tangent space $T_vE$. Together, these
isomorphisms induce a canonical bundle isomorphism between $TE$ and
$E\times E$. If $(v,w)\in E\times E$, we shall denote by $(v;w)$ the
corresponding tangent vector in $T_vE$. If $F:U\subseteq E\to E$ is
an analytic map, we denote $D_vF:T_vE\to T_{F(v)}E$ the derivative
of $F$ at $v\in U$. We denote by $DF:TU\to TX$ the bundle map
$(v;w)\mapsto D_vF(v;w)$. We shall denote by $F'(v;w)$ the vector in
$E$ corresponding to $DF(v;w) \in T_{F(v)}E$.

\begin{definition}
Let $\rad$ and $\theta_1$, \ldots, $\theta_p$ be the linear vector
fields $E\to TE$ defined by
\[\forall v\in E,\quad \rad(v) \eqdef  (v;v)\quad\text{and}\quad \theta_j(v) \eqdef  \bigl(v;v_j\bigr).\]
\end{definition}

Note that $\rad = \theta_1 + \cdots + \theta_p$.

\begin{definition}
A vector field $\xi$ is {\em asymptotically radial} if $\xi$ is
defined and analytic near $0\in E$ with $\xi(v)\underset{v\to 0}=\rad(v)+o\bigl(\|v\|\bigr)$.\\
A $p$-tuple of vector fields $(\xi_1,\ldots,\xi_p)$ is {\em
admissible} if
\begin{itemize}
\item  for all $j\in [1,p]$, $\xi_j$ is defined and analytic near $0$ in $E$, $\xi_j$ is tangent to $E_j$, $\xi_j$ vanishes when $v_j=0$ and  $\xi_j(v)\underset{v\to 0}=\theta_j(v)+o\bigl(\|v\|\bigr)$, and

\item for all
$i\in  [1,p]$ and all $j\in [1,p]$, the vector fields $\xi_{i}$ and $\xi_{j}$ commute.
\end{itemize}
\end{definition}

The tangency condition in theorem \ref{theo_local} requires some
explanation. We say that an analytic vector field $\zeta$ on an open
set $U$ is tangent to an analytic set $A\subseteq U$ if $\zeta(a)$
belongs to the tangent space $T_aA$ for every $a$ in the smooth part
of $A$.
We say that a germ of an analytic vector field $\zeta$ at $0$ is
tangent to the germ of the postcritical set of $F$ if there is a
neighborhood $U$ of $0$ such that
\begin{itemize}
\item $F$ and $\zeta$ are defined and analytic on $U$,

\item $F:U\to F(U)\subseteq U$ is proper and

\item the vector field $\zeta$ is tangent to the critical
value set of $F^{\circ n}:U\to F^{\circ n}(U)$ for all $n\geq 1$.
\end{itemize}

When $F$ is not postcritically finite, the postcritical set of $F$
is not analytic and the third condition involves tangency to {\em a
priori} infinitely many analytic sets.


\subsection{Quasihomogeneous Maps}\label{qhmaps}

Let $L$ be a $\C$-linear space. If $H:L\to L$ is a homogeneous map
of degree $k\geq 1$, then
\begin{equation}\label{eq_euleridentity}
DH\circ \rad=k\cdot \rad\circ H.
\end{equation}
This is known as Euler's identity. In fact, the converse is true. If
$H:(L,0)\to (L,0)$ is a germ of an analytic map, then $H$ is the
germ of a homogeneous map of degree $k\geq 1$ if and only if
$DH\circ \rad=k\cdot \rad\circ H$ near $0$. We shall adapt this
result to our setting as follows.

\begin{definition}
A map $H:E\to E$ is quasihomogeneous of multidegree
$(k_1,\ldots,k_p)$ if there are homogeneous maps $H_j:E_j\to E_j$ of
degree $k_j$ with $H=H_1\oplus\cdots\oplus H_p$.
\end{definition}

The following two lemmas about quasihomogeneous maps will be used later. 

\begin{lemma}\label{lemma_graphs}
Let $X\subset E$ be such that near $0$, the set $H^{-1}X$ coincides with the graph of a function $\varphi:(E_j^\top,0)\to (E_j,0)$.  Then $X=H^{-1}X=E_j^\top$ near $0$. 
\end{lemma}
\begin{proof}
Due to the structure of $H=H_j\oplus H_j^\top:E_j\oplus E_j^\top\to E_j\oplus E_j^\top$, and due to the homogeneity of $H_j$, we have that 
\[
H(e^{{2\pi i}/{k_j}}v_j+v_j^\top)=H_j(e^{{2\pi i}/{k_j}}v_j)+H_j^\top(v_j^\top)=H_j(v_j)+H_j^\top(v_j^\top)=H(v_j+v_j^\top).
\]
So 
\[
v_j+v_j^\top\in H^{-1}X \quad\iff\quad e^{{2\pi i}/{k_j}}v_j+v_j^\top\in H^{-1}X.
\]
Since near $0$, the set $H^{-1} X$ coincides with the graph of $\varphi$, we have 
\[
v_j=\varphi(v_j^\top)\quad\iff\quad e^{{2\pi i}/{k_j}}v_j=\varphi(v_j^\top)
\]
for $v_j$ sufficiently close to $0$. 
Thus $\varphi\equiv e^{{-2\pi i}/{k_j}}\varphi$ and so $\varphi$ vanishes identically near $0$. This shows that $H^{-1} X=E_j^\top$ near $0$, which implies that  $X=E_j^\top$ near $0$.
\end{proof}

\begin{lemma}\label{lemma_quasihomogeneous}
Let $H:(E,0)\to (E,0)$ be a germ of an analytic map. Then, $H$ is
the germ of a quasihomogeneous map of multidegree $(k_1,\ldots,k_p)$
if and only if \[\forall j\in [1,p],\quad DH\circ \theta_j = k_j
\cdot \theta_j\circ H.\]
\end{lemma}

\begin{proof}
Assume $H=H_1\oplus\cdots\oplus H_p$ with $H_i:E_i\to E_i$
homogeneous of degree $k_i\geq 1$. Then, $H_i(v_i) =
\phi_i(v_i,\ldots,v_i)$ with $\phi_i:E_i^{k_i}\to E_i$ a symmetric
$k_i$-linear map and so, $H_i'(v_i;w_i) = k_i\cdot
\phi_i(v_i,\ldots,v_i,w_i)$. Thus, for all $v\eqdef v_1+\cdots +
v_p\in E$,
\[
H'(v;v_j) = \sum_{i=1}^p k_i\cdot
\phi_i\bigl(v_i,\ldots,v_i,\pi_i(v_j)\bigr)  = k_j\cdot
\phi_j(v_j,\ldots,v_j) = k_j\cdot H_j(v_j).
\]
This shows that $DH\circ \theta_j(v)= k_j\cdot \theta_j\circ H(v)$.

Conversely, assume $DH\circ \theta_j = k_j \cdot \theta_j\circ H$
for all $j\in [1,p]$. In other words, \[\forall j\in [1,p],\quad
H'(v;v_j) = k_j \cdot \pi_j\circ H(v).\] Let $(U_j\subseteq
E_j)_{j\in [1,p]}$ be neighborhoods of $0$ such that
\begin{itemize}
\item $H$ is analytic on $U\eqdef U_1+\ldots+ U_p$ and

\item if $v_j\in U_j$ and $|\lambda|\leq 1$, then $\lambda v_j\in U_j$
\end{itemize}
Let $v=v_1+\cdots+v_p$ be a point in $U$. Fix $j\in [1,p]$. The map
\[\chi_j:(\lambda_1,\ldots,\lambda_p)\mapsto \pi_j\circ H(\lambda_1
v_1+\ldots +\lambda_p v_p)\] is defined and analytic in a
neighborhood of the closed polydisk $\overline \D^p$. In addition,
for all $i\in [1,p]$,
\begin{align*}
\frac{\partial \chi_j}{\partial
\lambda_i}(\lambda_1,\ldots,\lambda_p) & = \pi_j \circ
H'(\lambda_1 v_1+\ldots +\lambda_p v_p;v_i) \\
& = \frac{1}{\lambda_i} \cdot \pi_j \circ H' (\lambda_1 v_1+\ldots +\lambda_p v_p;\lambda_i v_i)  \\
&= \frac{k_i}{\lambda_i} \cdot \pi_j \circ \pi_i\circ H (\lambda_1
v_1+\ldots +\lambda_p v_p).
\end{align*}
If $i\neq j$, then $\partial \chi_j/\partial \lambda_i = 0$, which
shows that $\chi_j$ only depends on $\lambda_j$. In addition
\begin{equation}\label{eq_phi} \frac{\partial \chi_j}{\partial
\lambda_j} = \frac{k_j}{\lambda_j}\cdot \pi_j\circ
\chi_j.\end{equation}

Let $H_j:(E_j,0)\to (E_j,0)$ be the restriction of $\pi_j\circ H$ to
$E_j$. The previous discussion shows that
\[\forall (\lambda_1,\ldots,\lambda_p)\in \overline \D^p,\quad \chi_j(\lambda_1,\ldots ,\lambda_p) =\chi_j(0,\ldots ,0,\lambda_j,0,\ldots,0) =
H_j(\lambda_j v_j).\] In particular, taking
$\lambda_1=\cdots=\lambda_p=1$, we have that $\pi_j\circ H(v) =
H_j(v_j)$, and so $H= H_1\oplus \cdots \oplus H_p$.

Finally, the differential equation (\ref{eq_phi}) implies that
\[\frac{\partial H_j(\lambda v_j)}{\partial \lambda} =
\frac{k_j}{\lambda} H_j(\lambda v_j).\] Consider the map
\[\psi:\lambda\mapsto \lambda^{-k_j}  H_j(\lambda
v_j)\] which is defined and analytic in a neighborhood of $\overline
\D$. Its derivative satisfies
\[\psi'(\lambda) = -k_j \lambda^{-k_j-1} H_j(\lambda v_j) +
\lambda^{-k_j}\frac{k_j}{\lambda} H_j(\lambda v_j) =0.\] As a
consequence, $\psi(\lambda) = \psi(1)$ for all $\lambda \in
\overline \D$. So $H_j(\lambda v_j) = \lambda^{k_j} H_j(v_j)$ and
$H_j$ is the germ of a homogeneous map of degree $k_j$.
\end{proof}

\subsection{Linearizability of analytic vector fields}\label{linvf}

The following result is due to Poincaré. We include a proof for
completeness.

\begin{lemma}[Poincaré]
Any asymptotically radial vector field $\xi$ is linearizable: there
is a germ of an analytic map $\Phi:(E,0)\to (E,0)$ such that
$\Phi(0)=0$, $D_0\Phi=\id$ and $D\Phi \circ \xi = \rad\circ \Phi$.
\end{lemma}

\begin{proof}
Let $J$ be defined in a neighborhood of $0$ in $E$ by
\[J(v) \eqdef \frac{{\rm Re}\bigl\langle \xi(v)|v\bigr\rangle}{\|v\|^2}.\]
Since $\xi$ is asymptotically radial, we have that
\[J(v) \underset{v\to 0}= 1 + {\cal O}\bigl(\|v\|\bigr).\] So,
there are constants $C$ and $r>0$ such that
\[\forall v\in B(0,r),\quad \bigl|J(v)-1\bigr|\leq C\|v\|\leq 1/2.\]
In particular, $\xi$ is outward pointing on the boundary of
$B(0,r)$. Let ${\cal F}_t(v)={\cal F}(t,v)$ be the flow of the
vector field $\xi$. Since $\xi$ is outward pointing on the boundary
of $B(0,r)$, the map ${\cal F}_t$ is defined and analytic on
$B(0,r)$ for all $t\leq 0$.

For $t\leq 0$, set $G_t\eqdef \log\|{\cal F}_t\|$. Then, for $t\leq
0$,
\[\frac{\partial G_t}{\partial t} = J\circ {\cal F}_t \geq 1/2\]
and so, $G_t\leq G_0+t/2$ and $\|{\cal F}_t\|\leq re^{t/2}$. In
addition,
\[\left|\frac{\partial (G_t-t)}{\partial t}\right| = |J\circ {\cal F}_t -1|\leq
C\bigl\|{\cal F}_t\bigr\| \leq Cr e^{t/2}.\] So,
\[(G_t-t)-(G_0-0)\leq \int_{t}^0 C re^{u/2}\  {\rm d}u \leq
2Cr\quad\text{and therefore}\quad e^{-t}\|{\cal F}_t\|\leq
re^{2Cr}.\]

For $t\leq 0$, let $\Phi_t:B(0,r)\to E$ be the map defined by
\[\Phi_t\eqdef  e^{-t}\cdot {\cal F}_t.\]
Note that $\Phi_t(0) = 0$, $D_0\Phi_t={\rm id}$. Since $D{\cal F}_t
= \xi\circ {\cal F}_t$, we see that
\[D\Phi_t\circ \xi = e^{-t} \cdot\xi\circ {\cal F}_t = e^{-t}\cdot \bigl(
\rad\circ {\cal F}_t + o\big(\|{\cal F_t\|}\bigr)\bigr) = \rad \circ
\Phi_t + o\bigl(\|\Phi_t\|\bigr).\] The previous estimates show that
the family $(\Phi_t)_{t\leq 0}$ is uniformly bounded on $B(0,r)$ by
$r e^{2Cr}$. Thus, it is normal. Any limit value $\Phi$  as $t\to
-\infty$ linearizes $\xi$:
\[D\Phi\circ \xi = \rad\circ \Phi.\qedhere\]
\end{proof}

\begin{corollary}\label{coro_simultlin}
Let $(\xi_1,\ldots,\xi_p)$ be an admissible $p$-tuple of germs of
vector fields. Then, the $\xi_j$ are simultaneously linearizable,
\ie{}  there is a germ of an analytic map $\Phi:(E,0)\to (E,0)$
such that $\Phi(0)=0$, $D_0\Phi=\id$ and $D\Phi \circ \xi_j =
\theta_j\circ \Phi$ for all $j\in [1,p]$.
\end{corollary}

\begin{proof}
The germ of the vector field $\xi\eqdef \xi_1 + \cdots + \xi_p$ is
asymptotically radial, thus linearizable. Let $\Phi:(E,0)\to (E,0)$
be the linearizer. The vector fields $D\Phi\circ \xi_j$ commute with
$D\Phi \circ\xi = \rad$. It follows that they are linear vector fields
and so, $D\Phi \circ\xi_j = \theta_j$.
\end{proof}

\subsection{Liftable Vector Fields}\label{liftvf}

\begin{definition}
Let $F:U\to V$ be an analytic map. An analytic vector field $\xi$ on
$V$ is {\em liftable} if there is an analytic vector field $\zeta$
on $U$ which satisfies $DF \circ \zeta = \xi\circ F$. We say that
$\zeta$ lifts $\xi$.
\end{definition}

The critical point set of $F$ is the set ${\cal C}_F$ of points
$x\in U$ for which $D_xF$ is not invertible. The critical value set
of $F$ is ${\cal V}_F\eqdef F({\cal C}_f)$. If $\xi$ is a liftable
vector field on $V$ and if $\zeta$ lifts $\xi$, then for all $x\in
U-{\cal C}_F$, we have
\[\zeta(x) = (D_xF)^{-1} \xi\circ F(x).\]
Thus, when $F$ has discrete fibers, a liftable vector field $\xi$ on $V$ admits
a unique lift $\zeta$ on $U$ and we shall use the notation
\[F^* \xi \eqdef \zeta.\]



\begin{lemma}[Arnol'd]
Let $F:U\to V$ be an analytic map with discrete fibers. Let $\xi$ be
a vector field which is analytic on $V$ and tangent to the critical
value set ${\cal V}_F$. Then, $\xi$ is liftable and $F^* \xi$ is
tangent to ${\cal C}_F$.
\end{lemma}

\begin{proof}
The vector field $F^*\xi$ is well defined outside the critical point set.
According to a lemma of Hartogs, it is enough to show that $F^*\xi$
extends analytically outside a subset of $U$ of codimension $2$ in
order to know that it extends globally.

Since $F$ has discrete fibers, the critical point set ${\cal C}_F$ is
either empty or has codimension $1$. Thus, outside a codimension $2$
subset of $U$, the critical point set ${\cal C}_F$ is smooth. Moreover, it
follows from the Constant Rank Theorem that for generic $x\in{\cal
C}_F$, the kernel of $D_xF$ does not intersect the tangent space to
${\cal C}_F$ at $x$.

Thus, near generic points in ${\cal C}_F$, the map $F$ may be
locally expressed as
\[(x_1,\ldots,x_{m-1},x_m)\mapsto
(y_1,\ldots,y_{m-1},y_m) = (x_1,\ldots,x_{m-1},x_m^k)\] for some
integer $k\geq 2$. Since the vector field $\xi$ is tangent to the
critical value set $\{y_m=0\}$, it is of the form
\[\xi =\xi_1\frac{\partial }{\partial y_1} + \ldots
+\xi_{m-1}\frac{\partial }{\partial y_{m-1}} + y_m \xi_m
\frac{\partial }{\partial y_m}\] and \begin{align*}F^*\xi &=
(\xi_1\circ F) \frac{\partial }{\partial x_1} + \ldots
+(\xi_{m-1}\circ F) \frac{\partial }{\partial x_{m-1}} +
\left(\frac{x_m^k}{kx_m^{k-1}}
\xi_m\circ F\right) \frac{\partial }{\partial x_m}\\
&= (\xi_1\circ F) \frac{\partial }{\partial x_1} + \ldots
+(\xi_{m-1}\circ F) \frac{\partial }{\partial x_{m-1}} +
\left(\frac{1}{k} x_m \xi_m\circ F \right) \frac{\partial }{\partial
x_m}\end{align*} which clearly extends analytically through the
critical set $\{x_m=0\}$ and is tangent to the critical set.
\end{proof}

\begin{lemma}\label{newlemma}
Let $F:U\to V$ be an analytic map with discrete fibers, let $\xi$ be a liftable vector field on $V$, and let $\zeta\eqdef F^*\xi$ be its lift to $U$. Let $\phi(t,z)$ be the flow of $\xi$, and $\psi(t,z)$ be the flow of $\zeta$. 
\begin{enumerate}
\item\label{newlemma1} For all $z\in U$ and for $t$ sufficiently small, we have 
\[
F(\psi(t,z))=\phi(t,F(z)).
\]

\item\label{newlemma2} If $\xi\circ F (z)=0$, then $\zeta(z)=0$.

\item\label{newlemma3} If $\xi$ is tangent to an analytic set $A\subseteq U$, then $\zeta$ is tangent to $F^{-1}(A)$. 
\end{enumerate}

\end{lemma}
\begin{proof}
If $z\in U-{\cal{C}}_F$, then $F:(U,z)\to (V,F(z))$ is a local isomorphism.  It sends the vector field $\zeta$ to the vector field $\xi$. Thus it conjugates their flows, and the equality in part (\ref{newlemma1}) holds for $z\in U-{\cal{C}}_F$ and $t$ is sufficiently small. If $z\in {\cal{C}}_F$, it holds by analytic continuation (with respect to $z$).  

Parts (\ref{newlemma2}) and (\ref{newlemma3}) follow immediately since when the flow of $\xi$ preserves an analytic set $A$ (which may be reduced to a point if $\xi$ vanishes at this point), then the flow of $\zeta$ preserves the analytic set $F^{-1}(A)$. 
\end{proof}

We shall now study how admissible $p$-tuple of vector fields behave
under pullback.

\begin{lemma}\label{lemma_liftadmissible}
Assume $F:(E,0)\to (E,0)$ is an analytic germ having an adapted
superattracting fixed point. Let $(\xi_1,\ldots,\xi_p)$ be an
admissible $p$-tuple of liftable vector fields. Then, $(k_1\cdot F^*
\xi_1,\ldots, k_p \cdot F^* \xi_p)$ is an admissible $p$-tuple
of vector fields.
\end{lemma}

\begin{proof}
Fix $j\in [1,p]$, and set $\zeta_j\eqdef k_j \cdot F^*\xi_j$. Since the vector field $\xi_j$ is tangent to $E_j$, the vector field $\zeta_j$ is tangent to $F^{-1}(E_j)$. According to lemma \ref{prop_loctotinv}, $F^{-1}(E_j)$ coincides with $E_j$ in a neighborhood of $0$. Thus, the vector field $\zeta_j$ is tangent to $E_j$.  

Since the vector field $\xi_j$ vanishes on $E_j^\top\eqdef\{v_j=0\}$, the vector field $\zeta_j$ vanishes on $F^{-1}(E_j^\top)$. According to lemma \ref{prop_loctotinv}, $F^{-1}(E_j^\top)$ coincides with $E_j^\top$ in a neighborhood of $0$. So $\zeta_j$ vanishes on $E_j^\top$. 

Since $\zeta_j$ vanishes at $0$, we may write  
\[
\zeta_j(v)=\tau_j(v)+o\bigl(\|v\|\bigr)\quad \text{with}\quad \tau_j(v)\eqdef \bigl(v;A_j(v)\bigr)
\]
for some linear map $A_j:E\to E$. It remains to prove that $\tau_j=\theta_j$ which amounts to showing that $A_j=\pi_j$. 
Since $\zeta_j$ vanishes on $E_j^\top$, the linear map $A_j$ vanishes on $E_j^\top$. Thus it suffices to show that the restriction of $A$ to $E_j$ is the identity.

The map $F$ restricts to a self-map $F_j:E_j\to E_j$. The vector fields $\xi_j$ and $\zeta_j$ restrict to vector fields on $E_j$ (because they are tangent to $E_j$). It suffices to show that  $A_j(v_j)=v_j$. We therefore restrict our analysis to the space $E_j$, omitting the index $j$:
\begin{itemize}
\item $F(v)=H(v)+o\bigl(\|v\|^k\bigr)$ with $H:E\to E$ a nondegenerate homogeneous map of degree $k$,
\item $\xi(v)=(v;v)+o\bigl(\|v\|\bigr)$, 
\item $\zeta(v)=\tau(v)+o\bigl(\|v\|\bigr)$, and
\item $DF\circ \zeta = k\cdot\xi\circ F$.
\end{itemize}
On the one hand, 
\begin{align*}
DF\circ \zeta(v)=\bigl(F(v);  F'\circ\zeta(v)\bigr)&=\Bigl(  F(v);   H'\circ \zeta(v)+o\bigl(\|v\|^{k-1}\cdot \|\zeta(v)\|\bigr)\Bigr)\\
&=\Bigl(  F(v);   H'\circ \tau(v)+o\bigl(\|v\|^{k}\bigr)\Bigr)
\end{align*}
On the other hand, 
\[
k\cdot\xi\circ F(v)=\Bigl(F(v);k\cdot F(v)+o\bigl(\|F(v)\|\bigr)\Bigr)=\Bigl(F(v);k\cdot H(v)+o\bigl(\|v\|^k\bigr)\Bigr).
\]
It follows that $H'\circ \tau(v)=k\cdot H(v)$, thus 
\[
DH\circ \tau (v)=k\cdot\bigl(H(v);H(v)\bigr).
\]
According to Euler's identity, we therefore have $\tau(v)=(v;v)$ which implies that $A(v)=v$ as required. Lastly, the vector fields $\zeta_i$ and $\zeta_j$ commute for all $i,j\in[1,p]$ since the vector fields $\xi_i$ and $\xi_j$ commute for all $i,j\in [1,p]$. 
\end{proof}

\subsection{Dynamical Green Functions}\label{green}

We shall use dynamical Green functions introduced by Hubbard and
Papadopol \cite{hp}. We will first recall the construction for
homogeneous maps, and then explain how this construction may be
adapted to our setting.

Let $H:L\to L$ be a nondegenerate homogeneous map of degree $k$ on a
$\C$-linear space $L$ of dimension $m$. Then, the function
\[u_H:v\mapsto  \frac{1}{k}\log\bigl\|H(v)\bigr\| - \log\|v\|\]
is defined and bounded on $L-\{0\}$. It follows that the sequence of
plurisubharmonic functions
\[{\cal G}_H^n\eqdef \frac{1}{k^n}\log \|H^{\circ n}\|= {\cal G}_H^0 + \sum_{i=0}^{n-1}\frac{u_H\circ
F^{\circ i}}{k^i} :L\to \R\cup \{-\infty\}\] converges uniformly on
$L$ to a plurisubharmonic function ${\cal G}_H:L\to \R\cup
\{-\infty\}$ which is continuous on $L-\{0\}$ and satisfies
\[{\cal G}_H(v) = \log\|v\| +
{\cal{O}}(1)\quad\text{and}\quad {\cal G}_H\circ H = k\cdot {\cal G}_H.\] We
shall adapt this construction to our setting as follows.

\begin{lemma}\label{lemma_Green}
Let $F:(E,0)\to (E,0)$ be an analytic germ having an adapted
superattracting fixed point. There is a neighborhood $U$ of $0$ in $E$ such that for all
$j\in [1,p]$, the sequence of functions
\[{\cal G}_{j}^n \eqdef \frac{1}{k_j^n} \log \bigl\|\pi_j \circ F^{\circ n}\bigr\|:U\to \R\cup\{-\infty\}\]
 converges locally uniformly in $U$ to a plurisubharmonic function ${\cal G}_{j}:U\to \R\cup \{-\infty\}$ satisfying,
\[{\cal G}_{j}(v)\underset{v\to 0}= \log \bigl\|\pi_j(v)\bigr\|  + {\cal O}(1)
\quad\text{and}\quad {\cal G}_{j} \circ F = k_j\cdot  {\cal
G}_{j}.\]
\end{lemma}

\begin{proof}
Let $U\subset E$ be a sufficiently small neighborhood of $0$ so that
$F$  is defined and analytic on $U$, $F(U)\subseteq U$
and $U$ is contained in the basin of attraction of $0$. The functions
\[{\cal G}_{j}^n \eqdef \frac{1}{k_j^n} \log \bigl\|\pi_j \circ F^{\circ n}\bigr\|:U\to \R\cup\{-\infty\}\]
are then defined and plurisubharmonic on $U$. 

By assumption
\[
F_j(v)=H_j(v_j)+{\cal O}\bigl(\|v\|\cdot\|v_j\|^{k_j}\bigr).
\]
Since $H_j$ is homogeneous of degree $k_j$ and nondegenerate, $\|v_j\|^{k_j}={\cal{O}}\bigl(\|H_j(v_j)\|\bigr)$. 
As a consequence, 
\[
\bigl\|\pi_j\circ F(v)\bigr\|=\bigl\|F_j(v)\bigr\|\underset{v\to 0}\sim \bigl\|H_j(v_j)\bigr\|=\bigl\|H_j\circ \pi_j(v)\bigr\|
\]
and restricting $U$ if necessary, the function 
\[\log \|\pi_j\circ F\| - \log\|H_j\circ \pi_j\|\]
is defined and bounded in $U-E_j^\top$. 
The function $\ds \frac{1}{k_j}\log \| H_j\circ \pi_j\| - \log\|\pi_j\|$ is defined and bounded in $E-E_j^\top$ since $H_j$ is homogeneous of degree $k_j$ and nondegenerate. It follows that the function
\[u_j\eqdef \frac{1}{k_j}\log \|\pi_j \circ F\| - \log\|\pi_j\|\]
 is defined and bounded in $U-E_j^\top$. 

Now,  the sequence
\[{\cal G}_{j}^N = {\cal G}_j^0 + \sum_{n=0}^{N-1}\frac{u_j\circ
F^{\circ n}}{k_j^n} :U\to \R\cup \{-\infty\}\] converges uniformly
on $U$ to a plurisubharmonic function ${\cal G}_j$ whose difference with $ {\cal G}_j^0 =\log\|\pi_j\|$ is bounded  
as required.
\end{proof}


\subsection{Cartan's Lemma}\label{cartan}

We shall use the following lemma of Cartan which is a
multidimensional version of the Schwarz Lemma. We include a proof
for completeness.

\begin{lemma}[Cartan]
Let $V$ be a bounded connected open subset of $E$ containing $0$,
let $\Phi:V\to V$ be an analytic map such that $\Phi(0)=0$ and $D_0
\Phi ={\rm id}$. Then, $\Phi={\rm id}$.
\end{lemma}

\begin{proof}
The iterates $\Phi^{\circ n}$ are defined on $V$ for all $n\geq 0$.
For $n\geq 1$, let $\Psi_n:V\to E$ be defined as the average
\[\Psi_n \eqdef \frac{1}{n}\sum_{j=0}^{n-1} \Phi^{\circ j}.\]
Then,
\[\Psi_n(0) = 0,\quad D_0\Psi_n = {\rm id}\quad\text{and}\quad
\Psi_n\circ \Phi =  \Psi_{n} + \frac{\Phi^{\circ n}-{\rm id}}{n}.\]
In addition, the sequence $(\Psi_n)_{n\geq 1}$ is normal. Let $\Psi$
be a limit value. Then,
\[\Psi(0) = 0,\quad D_0\Psi = {\rm id}\quad\text{and}\quad \Psi\circ \Phi = \Psi.\]
In particular, $\Psi$ is invertible at $0$ and $\Phi$ is equal to
the identity near $0$, thus in $V$ by analytic continuation.
\end{proof}

\begin{corollary}
Let $V$ and $W$ be bounded connected open subsets of $E$ containing $0$,
let $\Phi:V\to W\subset E$ be an isomorphism and $\Psi:V\to W$ be an
analytic map such that $\Psi(0)=\Phi(0)$ and $D_0\Psi = D_0\Phi$.
Then, $\Psi = \Phi$.
\end{corollary}

\begin{proof}
Apply Cartan's lemma to $\Phi^{-1}\circ \Psi:V\to V$.
\end{proof}

\subsection{Proof of theorem \ref{theo_local}}\label{proofthm1}

(1) $\Rightarrow$ (3). For all $n\geq 1$, $H^{\circ n} = H_1^{\circ
n}\oplus \cdots \oplus H_p^{\circ n}$ and the critical value set of
$H^{\circ n}$ is \[{\cal V}_{H^{\circ n}}={\cal V}_{H_1^{\circ n}} +
\cdots + {\cal V}_{H_p^{\circ n}}.\] For all $j\in [1,p]$ and all
$n\geq 1$, the critical value set of $H_j^{\circ n}:E_j\to E_j$ is
homogeneous (\ie{} a complex cone with vertex at the origin). Thus,
for all $j\in [1,p]$, the vector field $\theta_j$ is tangent to the
critical value set of $H^{\circ n}$ and the vector field
$\zeta_j\eqdef \Phi^* \theta_j$ is tangent to the critical value
locus of $F^{\circ n}$. 

For all $i,j$, the vector fields $\theta_i$ and $\theta_j$ commute, thus $\zeta_i$ and $\zeta_j$ commute. Since $D_0\Phi=\id$ the linear part of $\zeta_j$ at $0$ is $\theta_j$. Since the vector field $\theta_j$ is tangent to $E_j$ and vanishes on $E_j^\top$, the vector field $\zeta_j$ is tangent to $\Phi^{-1}(E_j)$ and vanishes on $\Phi^{-1}(E_j^\top)$. 

We claim that for all $j\in [1,p]$, we have that $\Phi(E_j^\top)=E_j^\top$ near $0$. Indeed, recall that $E_j^\top$ is locally totally invariant by $F$ (see lemma \ref{prop_loctotinv}). So $X\eqdef \Phi(E_j^\top)$ is locally totally invariant by $H$. Since $D_0\Phi=\id$, $X$ and thus $H^{-1} X$ is locally the graph of a function $\varphi: E_j^\top\to E_j$. According to lemma \ref{lemma_graphs}, $X=H^{-1} X=E_j^\top$. 

It follows that $\Phi^{-1}(E_j^\top)=E_j^\top$ near $0$. In particular, $\Phi^{-1}(E_j)=E_j$ near $0$ (because $E_j$ is the intersection of all the $E_i^\top$ for $i\neq j$). Consequently, $\zeta_j$ is tangent to $E_j$ and vanishes on $E_j^\top$. Thus the $p$-tuple of vector fields $(\zeta_1,\ldots,\zeta_p)$ is admissible.

(2) $\Rightarrow$ (1). According to corollary \ref{coro_simultlin},
since $(\xi_1,\ldots,\xi_p)$ is admissible, there exists a germ of
an analytic map $\Phi:(E,0)\to (E,0)$ such that $D_0\Phi=\id$ and
$D\Phi \circ \xi_j = \theta_j\circ \Phi$ for all $j\in [1,p]$. Then,
$\Phi$ conjugates $F$ to a map $\check{F}$, defined and analytic
near $0$ in $E$, satisfying
\[D\check{F}\circ \theta_j=k_j\cdot\theta_j\circ \check{F}.\]
According to lemma \ref{lemma_quasihomogeneous}, $\check{F}$ is
quasihomogeneous with multidegree $(k_1,\ldots,k_p)$. Since
$D_0\Phi=\id$ we have that $\check{F}=H$.

(3) $\Rightarrow$ (2). Let $U_0$ be a sufficiently small
neighborhood of $0$ in $E$  so that
\begin{itemize}
\item $F$ is defined and analytic on $U_0$, $F(U_0)\subseteq
U_0$, and $U_0$ is contained in the attracting basin of $0$,

\item the vector field $\zeta_j$ is defined and analytic on $U_0$, 
tangent to the critical value set of $F^{\circ n}:U_0\to F^{\circ
n}(U_0)$ for all $n\geq 0$ and all $j\in [1,p]$.
\end{itemize}
Then, $\zeta_j$ is liftable by $F^{\circ n}$ and we may define a
holomorphic vector field $\zeta_j^n$ on $U_0$ by
\[\zeta_j^n\eqdef k_j^n \cdot  (F^{\circ n})^* \zeta_j.\]
Then, for all $n\geq 0$, we have that $DF\circ \zeta_j^{n+1} =
k_j\cdot \zeta_j^n\circ F.$ According to lemma
\ref{lemma_liftadmissible}, the $p$-tuple of vector fields
$(\zeta_1^n,\ldots,\zeta_p^n)$ is admissible. We will show that
there is a neighborhood $V$ of $0$ in $E$ on which the sequence of
vector fields $(\zeta_j^n)_{n\geq 0}$ converges uniformly to a
vector field $\xi_j$ for all $j\in [1,p]$. Then, the $p$-tuple of
vector fields $(\xi_1, \ldots,\xi_p)$ is admissible and satisfies
$DF\circ \xi_j =k_j\cdot \xi_j \circ F$ for all $j\in [1,p]$. The
proof will then be completed.

According to lemma \ref{lemma_Green}, the
sequence of functions
\[{\cal G}_{F,j}^n\eqdef \frac{1}{k_j^n} \log \|\pi_j\circ F^{\circ n}\|:U_0\to \R\cup\{-\infty\}\] converges
locally uniformly in $U_0$ to a function  ${\cal G}_{F,j}:U_0\to
\R\cup\{-\infty\}$ which is plurisubharmonic and satisfies
\[{\cal G}_{F,j}(v) \underset{v\to 0}=
\log\bigl\|\pi_j(v)\bigr\| + {\cal{O}}(1) \quad\text{and}\quad
{\cal G}_{F,j}\circ F = k_j\cdot {\cal G}_{F_j}.\] We set
\[{\cal G}_F^n \eqdef \max_{j\in [1,p]} {\cal
G}_{F,j}^n\quad\text{and}\quad {\cal G}_F \eqdef \max_{j\in [1,p]}
{\cal G}_{F,j}.\] Note that these functions are plurisubharmonic in
$U_0$ and take the value $-\infty$ only at $0$. In addition, the
sequence of functions ${\cal G}_F^n$ converges locally uniformly to
${\cal G}_F$ in $U_0$. In particular, if $M>0$ is sufficiently
large, the level sets $\{{\cal G}_F^n<-M\}$ are compactly contained
in $U_0$. From now on, we assume that $M>0$ is sufficiently large so
that the sets
\[V_n\eqdef \bigl\{v\in U_0~:~\forall j\in [1,p],~{\cal G}_{F,j}^n(v)<-M\bigr\}\]
are compactly contained in $U_0$ for all $n\geq 0$.

Similarly, the sequence of plurisubharmonic functions
\[{\cal G}_{H,j}^n \eqdef
\frac{1}{k_j^n} \log \| \pi_j\circ H^{\circ n}\|:E\to
\R\cup\{-\infty\}\] converges locally uniformly in $E$ to a function
${\cal G}_{H_j}:E\to \R\cup\{-\infty\}$ which is plurisubharmonic
 and satisfies
\[{\cal G}_{H,j}(v) \underset{v\to 0}=
\log\bigl\|\pi_j(v)\bigr\| + {\cal{O}}(1)\quad\text{and}\quad{\cal
G}_{H_j}\circ H = k_j\cdot {\cal G}_{H,j}.\] We set
\[{\cal G}_H^n\eqdef \max_{j\in [1,p]} {\cal G}_{H,j}^n, \quad {\cal G}_H\eqdef \max_{j\in [1,p]} {\cal G}_{H,j}
\quad\text{and}\quad  W_n\eqdef \bigl\{v\in E~:~{\cal
G}_{H}^n(v)<-M\bigr\}.\]

According to corollary \ref{coro_simultlin}, there are germs of
analytic maps $\Phi_n:(E,0)\to (E,0)$ such that $D_0\Phi_n=\id$ and
$D\Phi_n \circ \zeta_j^n = \theta_j\circ \Phi_n$ for all $n\geq 0$
and all $j\in [1,p]$. According to lemma
\ref{lemma_quasihomogeneous}, for all $n\geq 0$, the map
$\Phi_n\circ F \circ \Phi_{n+1}^{-1}:(E,0)\to (E,0)$ is
quasihomogeneous, thus equal to $H$. 

In other words, we have the
following commutative diagrams:
\[\diagram
(E,0)\rto^{\Phi_{n+1}}\dto_F & (E,0)\dto^H\\
(E,0)\rto_{\Phi_n} & (E,0)
\enddiagram
\quad\text{and}\quad 
\diagram
(E,0)\rto^{\Phi_{n}}\dto_{F^{\circ n}} & (E,0)\dto^{H^{\circ n}}\\
(E,0)\rto_{\Phi_0} & (E,0).
\enddiagram\]

\begin{lemma}
For $n\geq 0$, the linearizing map $\Phi_n$ is analytic on $V_n$.
\end{lemma}

\begin{proof}
Let us first show that the linearizing map $\Phi_n$ is defined on
$V_n$. Recall that $\Phi_n$ is defined as the linearizing map of the
asymptotically radial vector field
\[\zeta^n\eqdef \zeta_1^n + \cdots + \zeta_p^n.\]
Note that flowing along $\zeta^0_j$ during time $t<0$ decreases
${\cal G}_{F,j}^0$ by $t$. It follows that flowing along $(F^{\circ
n})^* \zeta^0_j$ during time $t<0$ decreases ${\cal G}_{F,j}^0\circ
F^{\circ n}$ by $t$. Thus, flowing along $\zeta^n_j$ during time
$t<0$ decreases ${\cal G}_{F,j}^n$ by $t$. Finally, since the vector
fields $(\zeta_j^n)_{j\in [1,p]}$ commute, flowing along $\zeta^n$
during time $t<0$ decreases ${\cal G}_F^n$ by $t$.

In particular, the flow of $\zeta^n$ is defined on $V_n$ for all
$t<0$, every trajectory remains in $V_n$ and converges to $0$. If we
denote by ${\cal F}_{n}$ this flow, the linearizing map $\Phi_n$ may
be obtained on $V_n$ as
\[\Phi_n(v) = \lim_{t\to -\infty} e^{-t}\cdot{\cal F}_{n}(t,x).\]
\end{proof}

Since $V_n$ is connected, the following commutative
diagram holds by analytic continuation to $V_n$:
\[\diagram
V_n \rto^{\Phi_n} \dto_{F^{\circ n}} & E \dto^{H^{\circ n}} \\
F^{\circ n}(V_n) \rto_{\Phi_0} & E.\enddiagram\]


\begin{lemma}\label{lemma_uj}
The function 
\[
u_j\eqdef \log\|\pi_j\circ\Phi_0\|-\log\|\pi_j\|
\]
is defined and bounded on $U_0-E_j^\top$.
\end{lemma}

\begin{proof}
Since $D_0\Phi_0=\id$, we have the a priori estimate
\[
\pi_j\circ\Phi_0(v)\underset{v\to 0}=\pi_j(v)+o\bigl(\|v\|\bigr).
\]

As in section \ref{proofthm1}, we show that that $X\eqdef \Phi_0(E_j^\top)$ coincides with $E_j^\top$ near $0$. Indeed, $E_j^\top$ is locally totally invariant by $F$. So $H^{-1} X= \Phi_1(E_j^\top)$ is a graph of a function $\varphi:(E_j^\top,0) \to (E_j,0)$ near $0$. According to lemma \ref{lemma_graphs}, $X=H^{-1} X=E_j^\top$. 

Consequently, $\pi_j\circ\Phi_0$ vanishes on $E_j^\top$, and so 
\[
\pi_j\circ\Phi_0(v)\underset{v\to 0}=\pi_j(v)+o\bigl(\|\pi_j(v)\|\bigr).
\]
The lemma follows immediately. 
\end{proof}

Observe that 
\begin{align*}
{\cal G}_{H,j}^n\circ \Phi_n &= \frac{1}{k_j^n}\log\left\|\pi_j\circ H^{\circ n}\circ \Phi_n\right\|\\
&=
\frac{1}{k_j^n}\log\left\|\pi_j\circ \Phi_0\circ F^{\circ n}\right\| =\frac{1}{k_j^n}u_j\circ F^{\circ n}+{\cal G}_{F,j}^n.
\end{align*}
Setting 
\[V\eqdef \bigl\{v\in U_0~:~{\cal
G}_{F}(v)<-M\bigr\}\quad\text{and}\quad W\eqdef \bigl\{v\in
U_0~:~{\cal G}_{H}(v)<-M\bigr\},\] 
we deduce that ${\cal G}_{H}^n\circ \Phi_n$ converges uniformly to ${\cal G}_{F}^n$ on every compact subset of $V$. Therefore the sequence $(\Phi_n)$ is uniformly bounded on every compact subset of $V$. 
Similarly, any compact subset of $W$ is contained in the image of $\Phi_n$ for $n$ large enough, and the sequence $(\Phi_n^{-1})$ is uniformly bounded on every compact subset of $W$.

Thus, the sequence of maps $(\Phi_n)$ is normal
on any compact subset of $V$ and the sequence of maps $(\Phi_n^{-1})$ is
normal on any compact subset of $W$. Extracting a subsequence, we
see that there is an isomorphism $\Phi:V\to W$ such that $\Phi(0)=0$
and $D_0\Phi={\rm id}$. According to Cartan's lemma, any limit value
of the sequence $(\Phi_n)$ must coincide with $\Phi$. Thus, the
whole sequence $(\Phi_n)$ converges to $\Phi$ locally uniformly in
$V$ and the sequence $(\zeta_j^n = \Phi_n^* \theta_j)$ converges to
$\xi_j=\Phi^*\theta_j$.

This completes the proof of theorem \ref{theo_local}.

\section{The Global Result}\label{globaltheoremsection}

We will now prove theorem \ref{theo_global}. 
Since $F:\Omega\to\Omega$ is proper and since ${\cal B}_a(F)$ is a connected component of $F^{-1}\bigl({\cal B}_a(F)\bigr)$, the restriction $F:{\cal B}_a(F)\to{\cal B}_a(F)$ is proper.

Let ${\cal
G}_{H,j}:E\to \R\cup\{-\infty\}$ and ${\cal G}_H:E\to
\R\cup\{-\infty\}$ be the dynamical Green functions of $H$
introduced above. Let ${\cal G}_{F,j}:{\cal B}_a(F)\to \R\cup\{-\infty\}$
be defined by
\[{\cal G}_{F,j}(x) = \frac{1}{k_j^n} {\cal G}_{H,j}\circ \Phi\circ
F^{\circ n}(x)\] where $n\geq 0$ is chosen sufficiently large so
that $F^{\circ n}(x)$ belongs to a neighborhood of $a$ on which $\Phi$ is defined. Let ${\cal
G}_F:{\cal B}_a(F)\to \R\cup\{-\infty\}$ be defined by
\[{\cal G}_F \eqdef \max_{j\in [1,p]} {\cal G}_{F,j}.\]
Let $M>0$ be sufficiently large so that $\Phi:(E,0)\to (E,0)$ has an inverse branch
 defined on
\[W \eqdef \bigl\{v\in E~:~{\cal G}_F(v)<-M\bigr\}.\]
Set $V\eqdef \Phi^{-1}(W)$, so that $\Phi:V\to W$ is an isomorphism.
Increasing $M$ if necessary, we see that $V$ is relatively compact
in ${\cal B}_a(F)$.

For $j\in [1,p]$, set
\[\xi_j\eqdef \Phi^*\theta_j\] which is defined and analytic near $0$. Then, $\xi_j = k_j\cdot F^*\xi_j$ near $a$.
Since $\theta_j$ is tangent to the postcritical set of $H$, and since near $a$, the map $\Phi$ sends the postcritical set of $F:{\cal B}_a(F)\to {\cal B}_a(F)$ to the postcritical set of $H:{\cal B}_0(H)\to {\cal B}_0(H)$, the vector field $\xi_j$ is tangent to the postcritical set of $F:{\cal B}_a(F)\to {\cal B}_a(F)$.
In particular, $\xi_j$ is tangent
to the critical value set of $F^{\circ n}$ for $n$ large enough. Thus we
can extend $\xi_j$ to the whole basin of attraction ${\cal B}_a(F)$ using the
formula $\xi_j = k_j^{n}\cdot (F^{\circ n})^* \xi_j$ for $n$ large
enough. We therefore have a vector field $\xi_j$ which is defined
and analytic on ${\cal B}_a(F)$ and satisfies $\xi_j = k_j\cdot F^*\xi_j$
on ${\cal B}_a(F)$. Let $(t,x) \mapsto {\cal F}_{t,j}(x)$ be the flow of
the vector field $\xi_j$ and let $(t,x)\mapsto {\cal F}_t(x)$ be the
flow of the vector field
\[\xi\eqdef \xi_1+\cdots+\xi_p.\]

\begin{lemma}
The map ${\cal F}_t$ is defined on ${\cal B}_a(F)$ for all $t\leq 0$. For
all $x\in {\cal B}_a(F)$, we have that ${\cal F}_t(x)\to 0$ as $t\to
-\infty$.
\end{lemma}

\begin{proof}
We first want to prove that ${\cal F}_t$ is defined on ${\cal B}_a(F)$ for
all $t\leq 0$. The maps ${\cal F}_{t,j}$ are defined and analytic on
$V$ for all $t\leq 0$. Indeed, $\Phi:V\to W$ conjugates ${\cal
F}_{t,j}$ to the linear map
\[{\cal H}_{t,j}:v_1+\cdots + v_p\mapsto v_1 + \cdots + v_{j-1} + e^t v_j + v_{j+1}
+ \cdots + v_p.\] The equality $\xi_j = k_j\cdot F^*\xi_j$ yields
$F\circ {\cal F}_{t,j} = {\cal F}_{k_jt,j}\circ F$ and so, $F^{\circ
n}\circ {\cal F}_{t,j} = {\cal F}_{k_j^nt,j}\circ F^{\circ n}$. For
all $x\in {\cal B}_a(F)$, if $n$ is large enough, ${\cal F}_{k_j^nt,j}\circ
F^{\circ n}(x)$ remains in a compact subset of ${\cal B}_a(F)$ for all
$t\leq 0$. Since $F^{\circ n}:{\cal B}_a(F)\to {\cal B}_a(F)$ is proper, ${\cal
F}_{t,j}(x)$ also remains in a compact subset of ${\cal B}_a(F)$, and so,
is defined for all $t\leq 0$. Since the vector fields $\xi_j$
commute, we have that
\[{\cal F}_t = {\cal F}_{t,1}\circ \cdots \circ {\cal F}_{t,p}.\]
So, ${\cal F}_t$ is defined on ${\cal B}_a(F)$ for all $t\leq 0$.

Near $0$, we have that ${\cal G}_{F,j} ={\cal G}_{H,j}\circ \Phi$
and $\Phi\circ {\cal F}_{t,j} = {\cal H}_{t,j}\circ \Phi$. In
addition, ${\cal G}_{H,j}\circ {\cal H}_{t,j} = {\cal G}_{H,j} +t$.
It follows that near $0$, we have that
\[{\cal G}_{F,j} \circ {\cal F}_{t,j} = {\cal G}_{F,j} + t.\]
The equality $\xi_j = k_j\cdot F^*\xi_j$ yields $F\circ {\cal
F}_{t,j} = {\cal F}_{k_jt,j}\circ F$. So, for $t\leq 0$, the
following equality is valid on ${\cal B}_a(F)$:
\begin{align*}
{\cal G}_{F,j} \circ {\cal F}_{t,j} = \frac{1}{k_j^n} {\cal G}_{F,j}
\circ F^{\circ n} \circ {\cal F}_{t,j}
&= \frac{1}{k_j^n} {\cal G}_{F,j} \circ {\cal F}_{k_j^nt,j}\circ F^{\circ n}\\
&= \frac{1}{k_j^n} {\cal G}_{F,j}\circ F^{\circ n} + t = {\cal
G}_{F,j} + t.
\end{align*}
As a consequence, for $t\leq 0$, the following equality is valid on
${\cal B}_a(F)$:
\[{\cal G}_F\circ {\cal F}_t = {\cal G}_F \circ{\cal F}_{t,1}\circ \cdots \circ {\cal F}_{t,p}
={\cal G}_F +t.\]

We now prove that $F^{-1}\{0\}=\{0\}$. As we have seen, each
trajectory $\bigl({\cal F}_t(x)\bigr)_{t\leq 0}$ remains in a compact
subset of ${\cal B}_a(F)$. Since ${\cal G}_F\circ {\cal F}_t = {\cal G}_F +
t$, as $t\to -\infty$, each trajectory $\bigl({\cal
F}_t(x)\bigr)_{t\leq 0}$ must converge to a point where ${\cal G}_F$
takes the value $-\infty$, \ie{} to a point in the backward orbit of
$0$. We may therefore partition ${\cal B}_a(F)$ in the basins of those
points (for the flow ${\cal F}_t$). The basins are open and since
${\cal B}_a(F)$ is connected, there is only one such basin: the basin of
$0$.
\end{proof}

The linearizer $\Phi$ of $\xi$ extends to the whole set ${\cal B}_a(F)$ by
\[\Phi(x) = e^{-t}\cdot \Phi\circ {\cal F}_t(x)\]
where $t\leq 0$ is chosen sufficiently negative so that ${\cal
F}_t(x)\in V$.

It is injective on ${\cal B}_a(F)$. Indeed, assume $\Phi(x_1) =\Phi(x_2)$
with $x_1$ and $x_2$ in ${\cal B}_a(F)$. Choose $t\leq 0$ sufficiently
negative so that ${\cal F}_t(x_1)$ and ${\cal F}_t(x_2)$ belong to
$V$. Then,
\[\Phi\circ{\cal F}_t(x_1)=e^t
\Phi(x_1) = e^t \Phi(x_2)=\Phi\circ {\cal F}_t(x_2)\] Since
$\Phi:V\to W$ is an isomorphism, we have that ${\cal F}_t(x_1)
={\cal F}_t(x_2)$, and so,
\[x_1={\cal F}_{-t}\circ {\cal F}_t(x_1)
= {\cal F}_{-t}\circ {\cal F}_t(x_2)=x_2.\]

The equality $\Phi\circ F = H\circ \Phi$ holds on ${\cal B}_a(F)$ by
analytic continuation. This shows that $\Phi({\cal B}_a(F))$ is contained
in the basin of attraction ${\cal B}_0(H)$.

Finally, $\Phi:{\cal B}_a(F)\to{\cal B}_0(H)$ is proper, thus an isomorphism.
Indeed, let $K\subset {\cal B}_0(H)$ be a compact set and let $n\geq 0$
be sufficiently large so that $K_n\eqdef H^{\circ n}(K)\subset W$.
Since $\Phi:V\to W$ is an isomorphism, $\Phi^{-1}(K_n)\subset V$ is
compact. Since $F^{\circ n}:{\cal B}_a(F)\to {\cal B}_a(F)$ is proper,
$F^{-n}\bigl(\Phi^{-1}(K_n)\bigr)$ is compact. This compact set contains
$\Phi^{-1}(K)$ which is closed since $\Phi$ is continuous. It
follows that $\Phi^{-1}(K)$ is compact.

This completes the proof that $\Phi:{\cal B}_a(F)\to {\cal B}_0(H)$ is an
isomorphism.

\section{Applications}\label{applications}

In this section we apply our theorems to a family of postcritically
finite endomorphisms of projective space which arose in \cite{k2}
and were studied in  \cite{bekp}. 

Let $h:{\mathbb P}^m\to {\mathbb
P}^m$ be an endomorphism; that is, a map which is everywhere holomorphic (the map has no indeterminacy points). Let ${\cal{C}}_h$ be critical locus of $h$. We define
the {\em{postcritical locus}} of $h$ to be
\[
{\cal{P}}_h:=\bigcup_{n \geq1} h^{\circ n}({\cal{C}}_h).
\]
The endomorphism $h$ is {\em{postcritically finite}}  if ${\cal{P}}_h$ is algebraic. Equivalently (via a Baire category argument), each component of 
 ${\cal{C}}_h$ is either periodic, or preperiodic to a periodic
cycle of components in ${\cal{P}}_h$. Postcritically finite endomorphisms of
${\mathbb P}^m$ were first studied by Forn{\ae}ss and Sibony in
\cite{fs}, and by Ueda in \cite{ue2}.

\subsection{Constructing endomorphisms}

The following construction is a particular case of a more general construction in \cite{k2}. Let $I$ be a finite set of cardinality $m\geq 1$. Denote by $E$ the $\C$-vector space of functions $x:I\to \C$ whose average is $0$. For $x \in E$, we use the notation $x_i\eqdef x(i)$ and set 
\[
P_x(t)\eqdef \frac{m+1}{m}\; \sum_{j\in I} \;\int_{x_j}^t\; \prod_{i\in I} (w-x_i)\;\mathrm{d} w.
\]
The polynomial $P_x$ is the unique monic centered polynomial of degree $m+1$ whose critical points are the points $(x_i)_{i\in I}$, repeated according to their multiplicities, and for which the barycenter of the critical values $\bigl(P_x(x_i)\bigr)_{i\in I}$ is $0$. 

The function $y\eqdef P_x\circ x:I\to \C$ belongs to $E$ and satisfies 
\[
\forall i\in I\quad y_i=P_x(x_i).
\]
We denote by $H:E\to E$ the map defined by 
\[
H(x)\eqdef P_x\circ x.
\]
\begin{proposition}
For $m\geq 2$, the map $H:E\to E$ is a homogeneous map of degree $m+1$. For $m\geq 3$ it induces an endomorphism $h:\P(E)\to \P(E)$, where $\P(E)$ is the projective space associated to $E$ (isomorphic to $\P^{m-2}(\C)$). 
\end{proposition}
\begin{proof}
The polynomial $P_x$ depends analytically on $x\in E$, therefore $H$ is analytic. Since 
\begin{align*}
P_{\lambda x} (\lambda t)&=\displaystyle\frac{m+1}{m}\; \sum_{j\in I} \;\int_{\lambda x_j}^{\lambda t}\; \prod_{i\in I} (w-\lambda x_i)\;\mathrm{d} w\\
&\underset{w=\lambda v}=\displaystyle\frac{m+1}{m}\; \sum_{j\in I} \;\int_{x_j}^{t}\; \prod_{i\in I} (\lambda v-\lambda x_i)\;\mathrm{d} (\lambda v)=\lambda^{m+1} P_x(t),
\end{align*}
we have $H(\lambda x)=\lambda^{m+1} H(x)$, so $H$ is homogeneous of degree $m+1$. 

Assume $P_x(x_i)=0$ for all $i\in I$. This means that $P_x$ has only one critical value, namely $0$. Therefore $P_x$ has only one critical point. Thus $x$ is constant, and the average of $x$ is $0$, so we have $x=0$. This implies that $H^{-1}(0)=\{0\}$, and consequently, $H:E\to E$ induces an endomorphism $h:\P(E)\to \P(E)$. 
\end{proof}

The significance of these endomorphisms is that their fixed points correspond to polynomials with fixed critical points. More precisely, we have the following correspondence. 
\begin{proposition}\label{prop_fixedpoints}
If $x$ is a fixed point of $H:E\to E$, then $P_x$ is a monic centered polynomial of degree $m+1$ with fixed critical points. If $P$ is a monic centered polynomial of degree $m+1$ with fixed critical points, then there exists $x\in E$ such that $H(x)=x$, and $P_x=P$.
\end{proposition}
\begin{proof}
First let $x$ be a fixed point of $H$. The polynomial $P_x$ is monic, centered, and of degree $m+1$. The critical points of $P_x$ are the $x_i$, and $x_i=P_x(x_i)$ for all $i\in I$.

Conversely, let $P$ be a monic centered polynomial of degree $m+1$ with fixed critical points. The polynomial $P$ has exactly $m$ critical points (counted with multiplicity). Let $x$ be a surjection from $I$ to the critical set of $P$ such that for each critical point $c\in {\cal{C}}_P$, the cardinality of $x^{-1}(c)$ is the multiplicity of $c$ as a critical point of $P$. Then $x\in E$, and $P_x = P$, and  $x$ is a fixed point of $H$. 
\end{proof}

A set of particular interest is the noninjectivity locus  
\[
\Delta\eqdef \{x\in E :\exists\; i\neq j\text{ with }x_i=x_j\}.
\]
By proposition \ref{prop_fixedpoints}, the fixed points of $H$ correspond to polynomials with fixed critical points: the fixed points in $E-\Delta$ correspond to polynomials with simple critical points, whereas the fixed points in $\Delta$ correspond to polynomials with at least one multiple critical point. 

The locus $\Delta$ is a union of hyperplanes which are invariant by $H$; it is a stratified space where each stratum is invariant. More precisely, denote by $\Part$ the set of all partitions of $I$, and set $\part=\Part-\{I\}$. Let $\I\in\part$ be the singleton partition of $I$: that is 
\[
\I\eqdef\bigl\{\{i\}:i\in I\bigr\}.
\]
Given $\J\in\part$, let $L_\J\subseteq E$ be the linear space defined by 
\[
L_\J\eqdef\{x\in E : x \text{ is constant on each element of }\J\}.
\]
Note that $E=L_\I$ and the dimension of $L_\J$ is $|\J|-1$. We say that $\K\in\part$ is a contraction of $\J\in\part$ if all elements of $\K$ are unions of elements of $\J$. We denote this as $\K\preceq\J$, and if $\K\neq\J$, we use the notation $\K\prec\J$. When $\K$ is a contraction of $\J$, the linear space $L_\K$ is contained in $L_\J$, and the codimension of $L_\K$ in $L_\J$ is $|\J|-|\K|$. 
The stratification of $\Delta$ is given by 
\[
\Delta=\bigcup_{\J\in\part-{\cal I}} L_\J.
\]
If $x\in L_\J$, then $H(x)=P_x\circ x$ is constant on each element of $\J$. Thus $H(L_\J)\subseteq L_\J$. In particular, $h:\P(E)\to \P(E)$ restricts an endomorphism $h:\P(L_\J)\to\P(L_\J)$.
For $\J\in\part$, define the set 
\[
\Delta_\J\eqdef\bigcup_{{\K\in\part}\atop {\K\prec \J}} L_\K.
\]

\begin{remark}\label{newton_uniqueness}
Assume $x$ is a fixed point of $H$ in $E-\Delta$; then by proposition  \ref{prop_fixedpoints}, $P_x$ has $m$ fixed critical points. The polynomial $P_x$ is of degree $m+1$, so there is a unique (repelling) fixed point of $P_x$ which is not critical. Moreover, since $P_x$ is centered and $m+1\geq 3$, the fixed points of $P_x$ are also centered; this implies that this fixed point is at 0. 

Consequently, the rational map $f:w\mapsto 1/P_x(1/w)$ has degree $m+1$, a repelling fixed point at $\infty$, and $m+1$ superattracting fixed points. This map is therefore the Newton's method of a polynomial $Q$ of degree $m+1$. 

The critical points of $f$ are the zeroes of $Q$, and the zeroes of $Q''$. Since $f$ has a critical point of multiplicity $m$ at 0, $Q''(w)=aw^{m-1}$ for some $a\in\C^*$ and $Q$ vanishes at 0. An elementary computation then shows that 
\[
P_x(z)=\frac{m+1}{m} z+z^{m+1}.
\]
This polynomial is indeed one with simple critical points which are fixed. So $H$ has exactly $m!$ fixed points in $E-\Delta$. 
\end{remark}

\begin{figure}[h] 
   \centering
   \includegraphics[width=2.5in]{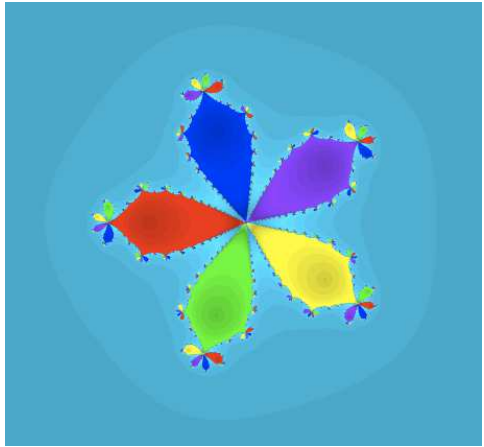} 
   \caption{The polynomial $P(z)=\frac{6}{5}z+z^6$ has five critical points, each of which is a superattracting fixed point of $P$. There is a repelling fixed point at 0. The polynomial $P$ is conjugate to a Newton's method for finding roots of some polynomial $Q$ as discussed in remark \ref{newton_uniqueness}.}
   \label{fig:example}
\end{figure}

\subsection{The endomorphisms are postcritically finite}

Our goal now is to prove that the endomorphisms $h:\P(L_\J)\to\P(L_\J)$ are postcritically finite: we will identify the critical locus of $H:L_\J\to L_\J$, and show that it is invariant. For this we will use the following observation. 

Let $J$ be a nonempty proper subset of $I$. Let $E_J$ be the set of functions $J\to \C$ whose average is $0$. There is a natural projection $\pi_J:E\to E_J$ given by 
\[
E\ni x\longmapsto x|_{J}-\text{Average}(x|_{J})\in E_J.
\]
Let $H_J:E_J\to E_J$ be the homogeneous map constructed as above with the set $J$ instead of the set $I$. 
\begin{lemma}\label{lemma_expansion}
Assume $\J\in\part$, $x\in L_\J-\Delta_\J$, and $J\in\J$. Then, as $v \to 0$ in $E$, we have the following expansion
\[
\pi_J\circ H (x+v)=C_J\cdot H_J\circ\pi_J(v)+ {\cal{O}}\bigl(\|v\|\cdot \|\pi_J(v)\|^{|J|+1}\bigr)\quad \text{with}\quad C_J\in \C-\{0\}.
\]
 \end{lemma}
\begin{proof}
Without loss of generality, assume that $m_J\eqdef |J|\geq 2$ since otherwise $E_J$ has dimension 0, and the result is vacuous. 
Set  $y\eqdef x+v$, $a_y\eqdef\text{Average}(y|_{J})$ and $z\eqdef\pi_J(v)=\pi_J(y)=y|_{J}-a_y$. Let $Q_z$ be the polynomial defined by
\[
Q_z(t)\eqdef \frac{m_J+1}{m_J}\; \sum_{j\in J} \;\int_{z_j}^t\; \prod_{i\in J} (w-z_i)\;\mathrm{d} w.
\]
Then
\[
P'_y(t)=(m+1)\prod_{i\in I} (t-y_i)\quad\text{and}\quad Q'_z(t-a_y)=(m_J+1)\prod_{i\in J} (t-y_i).
\]
Thus 
\[
P'_y(t)=Q'_z(t-a_y)\cdot R_y(t)\quad\text{with}\quad R_y(t)=\frac{m+1}{m_J+1}\prod_{i\in I-J} (t-y_i).
\]
Since $x\in L_\J-\Delta_\J$, for all $i\in I-J$ we have $x_i\neq a_x$, thus $R_x(a_x)\neq 0$.
For all $i\in J$ and $j\in J$, as $y\to x$, we have:
\begin{itemize}
\item $|y_j-y_i|=|z_j-z_i|={\cal{O}}\bigl(\|z\|\bigr)$,\\
\item $\underset{t\in [y_i,y_j]} \sup |Q_z'(t-a_y)|={\cal{O}}\bigl(\|z\|^{m_J}\bigr)$, and\\
\item $\underset{t\in [y_i,y_j]}\sup|R_y(t)-R_x(a_x)|={\cal{O}}\bigl(\|v\|\bigr)$.
\end{itemize}
Thus for all $i\in J$ and $j\in J$, 
\begin{align*}
 P_y(y_j)-P_y(y_i) & -\bigl(Q_z(z_j)-Q_z(z_i)\bigr)R_x(a_x)\\
 &=\int_{y_i}^{y_j} P'_y(t)-Q'_z(t-a_y)R_x(a_x)\;\mathrm{d}t\\
&=\int_{y_i}^{y_j} Q'_z(t-a_y)\bigl(R_y(t)-R_x(a_x)\bigr)\;\mathrm{d}t\in {\cal{O}}\bigl(\|v\|\cdot \|z\|^{m_J+1}\bigr).
\end{align*}
Setting $C_J\eqdef R_x(a_x)$, we deduce that for all $j\in J$, 
\begin{align*}P_y(y_j) - \frac{1}{m_J+1} \sum_{i\in J} P_y(y_i) &= C_J\cdot \Big(Q_z(z_j) - \frac{1}{m_J+1} \sum_{i\in J} Q_z(z_i)\Big) 
+ {\cal{O}}\bigl(\|v\|\cdot \|z\|^{m_J+1}\bigr)\\
&=C_J\cdot Q_z(z_j)+ {\cal{O}}\bigl(\|v\|\cdot \|z\|^{m_J+1}\bigr).
 \end{align*}
 (Recall that $\sum_{i\in J} Q_z(z_i)=0$ since $Q_z$ is centered).
The lemma follows since 
\[\pi_J\circ H = P_y\circ y|_{J} - \text{Average}(P_y\circ y|_{J})\quad \text{and}\quad H_J\circ \pi_J(y) = H_J(z) = Q_z\circ z.\qedhere\]
\end{proof}

\begin{proposition}\label{critprop}
For $\J\in\part$, the critical set of $H:L_\J\to L_\J$ is $\Delta_\J$, 
$H(\Delta_\J)=\Delta_\J$, and $h:\P(L_\J)\to\P(L_\J)$ is postcritically finite. 
\end{proposition}
\begin{proof}
It is enough to prove that the critical set of $H:L_\J\to L_\J$ is $\Delta_\J$ as the rest is an immediate consequence. Given $J_1,J_2\in\J$, set 
\[
K\eqdef J_1\cup J_2 \quad\text{and}\quad\K=\K(J_1,J_2)\eqdef \J-\{J_1,J_2\}\cup\{K\}.
\]
Note that $\K\prec\J$ and $|\K|=|\J|-1$, so that $L_\K$ has codimension 1 in $L_\J$. 
Choose a vector $v\in L_\J-L_\K$ so that $L_\J=L_\K\oplus\mathrm{Span}(v)$.  According to lemma \ref{lemma_expansion}, for all $x\in L_\K$ as $t\to 0$ we have 
\[
\pi_K\circ H(x+tv)=\cal{O}(t^{|K|+1}).
\]
It follows that the Jacobian of $H:L_\J\to L_\J$ vanishes with order at least $|K|$ along $L_\K$. 

Now 
\[
\Delta_\J=\bigcup_{J_1,J_2\in\J\atop J_1\neq J_2} L_{\K(J_1,J_2)}
\]
and $L_{\K(J_1,J_2)}$ is a critical component of $H:L_\J\to L_\J$ with multiplicity $|J_1|+|J_2|$. So $\Delta_\J$ is contained in the critical set of $H:L_\J\to L_\J$ as a hypersurface of total degree 
\[
\sum_{J_1,J_2\in\J\atop J_1\neq J_2} |J_1|+|J_2|=\bigl(|\J|-1\bigr)\cdot \sum_{J\in\J} |J|=\bigl(|\J|-1\bigr)\cdot m.
\]
Since $H:L_\J\to L_\J$ is a homogeneous map of degree $m+1$, its critical locus is a hypersurface of total degree $\mathrm{dim}(L_\J)\cdot m=\bigl(|\J|-1\bigr)\cdot m$. And therefore the critical set of $H:L_\J\to L_\J$ coincides with $\Delta_\J$ as required. 
\end{proof}
We now analyze the spectrum of $D_x H:T_x E\to T_x E$, where $x$ is a fixed point of $H:E\to E$. It turns out that for the fixed points in $E-\Delta$, we have a complete understanding of this spectrum as outlined in the example below. 

\begin{example} 
We now demonstrate by way of example that the eigenvalues of $D_x H$ at a fixed point $x\in E-\Delta$ are precisely $\lambda_k\eqdef (m+1)/k,k\in[1,m-1]$, with corresponding eigenspace $\mathrm{Span}(x^k)$.\footnote{Recall that $x:I\to \C$ is a function and $x^k$ is the function $i\mapsto x_i^k\in\C$.}

According to remark \ref{newton_uniqueness}, there is a unique polynomial 
\[
P(z)=\frac{m+1}{m}z+z^{m+1}
\]
 which is monic and centered, with simple fixed critical points, corresponding to a fixed point of $H$ in $E-\Delta$. Observe that if $c$ is a critical point of $P$, then $P''(c)=-(m+1)/c$. 

Now if $P_t(z)=P(z)+tQ(z)+o(t)$ with $Q(z)\in\C_{m-1}[z]$, and if $c_t=c+tv+o(t)$ is a critical point of $P_t$, then 
\[
0=P_t'(c_t)=P'(c)+t(P''(c)v+Q'(c))+o(t)
\]
so that 
\[
v=-\frac{Q'(c)}{P''(c)}=\frac{c}{m+1}Q'(c),
\]
and 
\[
P_t(c_t)=P(c)+tQ(c)+tP'(c)v+o(t)=c+tQ(c)+o(t).
\]
Therefore $v\in T_x E$ is an eigenvector associated to the eigenvalue $\lambda$ if and only if 
\[
\forall\; i\in I, \quad Q(x_i)=\lambda v_i=\frac{\lambda x_i}{m+1} Q'(x_i).
\]
This is clearly true if $Q(z)=z^k$, $\lambda=(m+1)/k$, and $v_i=x_i^k/\lambda$. 
\end{example}

\subsection{Fixed points are super-saddles}
\begin{proposition}\label{prop_spec}
Any fixed point of $H:E\to E$ is a super-saddle. More precisely, if $x$ is a fixed point of $H:E\to E$, then:
\begin{itemize}
\item $T_xE=\mathrm{Ker}(D_x H)\oplus \mathrm{Im}(D_x H)$
\end{itemize}
and if $\J\in\part$ and $x\in L_\J-\Delta_\J$ then:
\begin{itemize}
\item $\mathrm{Im}(D_x H)=T_xL_\J$, and 
\item the spectrum of $D_x H:T_xL_\J\to T_xL_\J$ is contained in $\C-\overline\D$ and therefore $x$ is a repelling fixed point of $H:L_\J\to L_\J$. 
\end{itemize}
\end{proposition}
\begin{proof}
According to proposition \ref{critprop}, $x$ is not a critical point of $H:L_\J\to L_\J$, thus the restriction $D_x H:T_xL_\J\to T_xL_\J$ is invertible. As a consequence the image of $D_x H$ contains $T_xL_\J$ and $\mathrm{Ker}(D_x H)\cap T_xL_\J=\{0\}$. The kernel of the projection 
\[
\pi\eqdef \sum_{J\in\J} \pi_J:T_xE\to \bigoplus_{J\in\J} T_0E_J
\]
is $T_xL_\J$. According to lemma \ref{lemma_expansion} $\pi_J\circ D_x H=0$ for all $J\in\J$. So $\pi\circ D_x H=0$, and the image of $D_x H$ is contained in $\mathrm{Ker}(\pi)=T_xL_\J$. This implies that $\mathrm{Im}(D_x H)=T_xL_\J$. In addition the codimension of $\mathrm{Ker}(D_x H)$ is the dimension of $T_xL_\J$ and since $\mathrm{Ker}(D_x H)\cap T_xL_\J=\{0\}$, we conclude that $T_xE=\mathrm{Ker}(D_x H)\oplus \mathrm{Im}(D_x H)$. 

Since $H:L_\J\to L_\J$ is homogeneous of degree $m+1$, there is an obvious eigenvalue $m+1$ associated to the eigenspace $\mathrm{Span}(x)$. 
Proposition \ref{prop:transpose} below asserts that we can endow $T_xL_\J$ with an appropriate norm so that  the linear map $(D_x H)^{-1}: T_xL_\J\to T_xL_\J$ is contracting. As a consequence,  the spectrum of $(D_x H)^{-1}: T_xL_\J\to T_xL_\J$ is contained in $\D$ and thus, the spectrum of the linear map $D_x H:T_xL_\J\to T_xL_\J$ is contained in $\C-\overline\D$. 
\end{proof}
By definition $E$ is a subspace of $\C^I$ of codimension 1. It is the kernel of the linear form 
\[
e^I:E\ni v\mapsto \frac{1}{|I|}\sum_{i\in I} v_i\in \C.
\]
As a consequence $E^*$ may be identified with the quotient $(\C^I)^*/\mathrm{Span}(e^I)$. 
Let ${\bf 1}:I\to \C$ be the function which is constant and equal to 1. Since $\C^I=E\oplus\mathrm{Span}({\bf 1})$, the dual space $E^*$ may also be identified with the orthogonal space 
\[
\bigl(\mathrm{Span}({\bf 1})\bigr)^\perp = \bigl\{\alpha\in(\C^I)^*:\alpha({\bf 1})=0\bigr\}.
\]
More generally, given $J\subseteq I$ let $e^J\in (\C^I)^*$ be the linear form defined by 
\[
e^J(v)= \frac{1}{|J|}\sum_{j\in J} v_j.
\]
Note that for $\J\in\part$, 
\[
E=L_\J\oplus \bigcap_{J\in\J} \mathrm{Ker}(e^J)
\]
and therefore $L_\J^*$ may be identified with
\[
\bigl(\mathrm{Span}({\bf 1})\bigr)^\perp \cap \mathrm{Span}(e^J;J\in\J)=\left\{\sum_{J\in\J}\lambda_Je^J:\lambda_J\in\C\text{ and }\sum_{J\in\J}\lambda_J=0\right\}.
\]
If $\displaystyle\alpha=\sum_{J\in\J} \lambda_J e^J\in L_\J^*$ then the pairing with $v\in L_\J$ is given by
\[
\alpha(v)=\sum_{J\in \J}\lambda_J v_J\quad\text{with}\quad v(J)=\{v_J\}.
\]

Given $x\in L_\J-\Delta_\J$ we will now identify $(T_x L_\J)^*$ with a space of quadratic differentials, equip this space with an appropriate norm then identify the transpose of $(D_x H)^{-1}:T_x L_\J\to T_xL_\J$, and finally show that this transpose is strictly contracting. So for $J\in\J$, let $q_J$ be the quadratic differential defined on $\C$ by
\[
q_J=\frac{dz^2}{z-x_J}.
\]
Set  
\[
{\cal Q}_x\eqdef \left\{\sum_{J\in\J} \lambda_J q_J:\sum_{J\in\J} \lambda_J=0\right\}.
\]
A quadratic differential $q=\sum \lambda_J q_J \in{\cal Q}_x$ may be paired with a tangent vector $v\in T_x L_\J$ as follows
\[
\left<q,v\right>\eqdef \sum_{J\in\J} \lambda_J v_J=\sum_{J\in\J} \mathrm{Res}_{x_J}(q\cdot \xi_v)
\]
where $\xi_v$ is any holomorphic vector field near $x(I)$ which takes the value $v_J$ at $x_J$. According to the previous discussion, this pairing gives an identification of $(T_x L_\J)^*$ with ${\cal Q}_x$. 

Choose $R$ large enough so that $P_x^{-1}(D_R)$ is compactly contained in $D_R$, where $D_R$ is the disk centered at 0 of radius $R$. We equip ${\cal Q}_x$ with the $L^1$ norm 
\[
\|q\|\eqdef \int_{D_R} |q|.
\]
\begin{proposition}\label{prop:transpose}
Assume $x\in L_\J-\Delta_\J$. The transpose of the linear mapping $(D_x H)^{-1}:T_x L_\J\to T_x L_\J$ is identified with the push-forward operator 
\[
(P_x)_*:{\cal Q}_x\ni q\mapsto \sum g^* q\in {\cal Q}_x
\]
where  $g$ ranges over the inverse branches of $P_x$. In addition,
\[
\forall \;q\in{\cal Q}_x\quad\;\bigl\|(P_x)_* q\bigr\| < \|q\|.
\]
\end{proposition}
\begin{proof}
The space ${\cal Q}_x$ is the set of meromorphic quadratic differentials on $\P^1(\C)$ which are holomorphic outside $x(I)$, have at most simple poles along $x(I)$ and at most a double pole at $\infty$. 
Set $P\eqdef P_x$. If $q\in {\cal Q}_x$ then $P_*q$ is a meromorphic quadratic differential on $\P^1(\C)$. Since $q$ has at most a double pole at $\infty$, $P_*q$ also has at most a double pole at $\infty$. The other poles of $P_* q$ are simple and contained in $P\bigl(x(I)\bigr)$ union the critical value set of $P$, that is $x(I)$. This shows that $P_*$ maps ${\cal Q}_x$ to ${\cal Q}_x$. In addition 
\[
\|P_* q\|=\int_{D_R}\left|\sum g^*q\right|\leq \int_{D_R}\sum |g^*q|=\int _{P^{-1}(D_R)} |q|<\int_{D_R}|q|=\|q\|.
\]
Therefore we only need to prove that for all $v\in T_x L_\J$ and all $q\in {\cal Q}_x$
\[
\left<q,v\right>=\bigl<P_* q,D_xH(v)\bigr>.
\]
Fix $v\in T_x L_\J$ and $q\in{\cal Q}_x$. 
Let $U$ be the complement in $D_R$ of pairwise disjoint closed disks centered at the points of $x(I)$, and contained in $D_R$. If $\xi$ is a $C^\infty$ vector field on $\C$ which is holomorphic outside $U$, vanishes outside $D_R$, and satisfies $\xi\circ x = v$, then
\[
\left<q,v\right>=\sum_{J\in\J} \mathrm{Res}_{x_J}(q\cdot\xi)=-\frac{1}{2\pi i} \int_{\partial U} q\cdot \xi \underset{\text{Stokes}}=\frac{1}{2\pi i }\int_U q\cdot \overline\partial \xi=\frac{1}{2\pi i }\int_\C q\cdot \overline\partial \xi.
\]
With an abuse of notation set 
\[
x_t\eqdef x+tv,\quad P_t\eqdef P_{x_t},\quad\text{and}\quad y_t\eqdef H(x_t)=P_t\circ x_t.
\]
Then, 
\[\dot x \eqdef \frac{\partial x_t}{\partial t}\Big|_{t=0} = v \quad\text{and}\quad
\dot y \eqdef \frac{\partial y_t}{\partial t}\Big|_{t=0} = D_xH(v).\]
In addition, the critical point set of $P_t$ is $x_t(I)$ and the critical value set of $P_t$ is $y_t(I)$. 
Let $\left(\varphi_t:\C\to\C\right)_{t\in(-\epsilon,\epsilon)}$ be an analytic family of $C^\infty$ diffeomorphisms such that 
\begin{itemize}
\item $\varphi_0=\mathrm{id}$,
\item $\varphi_t$ is the identity outside $D_R$,
\item $\varphi_t$ is holomorphic outside $U$ and
\item $y_t=\varphi_t\circ y$.
\end{itemize}
Note that
\[\dot\varphi\eqdef\frac{\partial\varphi_t}{\partial t}\Big|_{t=0}
\]
is a $C^\infty$ vector field on $\C$ which is holomorphic outside $U$, vanishes outside $D_R$, and satisfies $\dot\varphi\circ x = D_xH(v)$.
Since $\varphi_t$ follows the critical value set of $P_t$ we can lift the diffeomorphisms $\varphi_t:\C\to\C$ to diffeomorphisms $\psi_t:\C\to\C$ so that $\psi_0=\mathrm{id}$ and the following diagram commutes
\[
\xymatrix{
&(\C,x(I))\ar[r]^{\psi_t}\ar[d]^{P}  &(\C,x_t(I))\ar[d]^{P_t}\\
&(\C,y(I))\ar[r]^{\varphi_t}  &(\C,y_t(I))}
\]
Then, 
\[\dot\psi\eqdef\frac{\partial\psi_t}{\partial t}\Big|_{t=0}
\]
is a $C^\infty$ vector field on $\C$ which is holomorphic outside $P^{-1}(U)$, vanishes outside $D_R$, and satisfies $\dot\psi\circ x = v$.
In addition, the infinitesimal Beltrami differentials $\bar\partial \dot\varphi$ and $\bar\partial \dot \psi$ satisfy the relation 
\[\bar\partial \dot \psi=P^* (\bar\partial \dot\varphi).\]
Therefore
\[\left<q,v\right>=\int_\C q\cdot \bar\partial \dot\psi = \int_\C q\cdot P^*( \bar\partial \dot\varphi)
=\int_\C P_* q\cdot  \bar\partial \dot\varphi = \bigl<P_*q,D_xH(v)\bigr>\]
as required. 
\end{proof}
\subsection{Superattracting fixed points of $h:\P(E)\to\P(E)$} 

If $\J\in\part$ has cardinality 2, then $L_\J$ has dimension 1 and its image in $\P(E)$ is a point $a_\J$. This is a superattracting fixed point for $h:\P(E)\to\P(E)$. 

\begin{figure}[htbp] 
   \centering
   \includegraphics[width=5in]{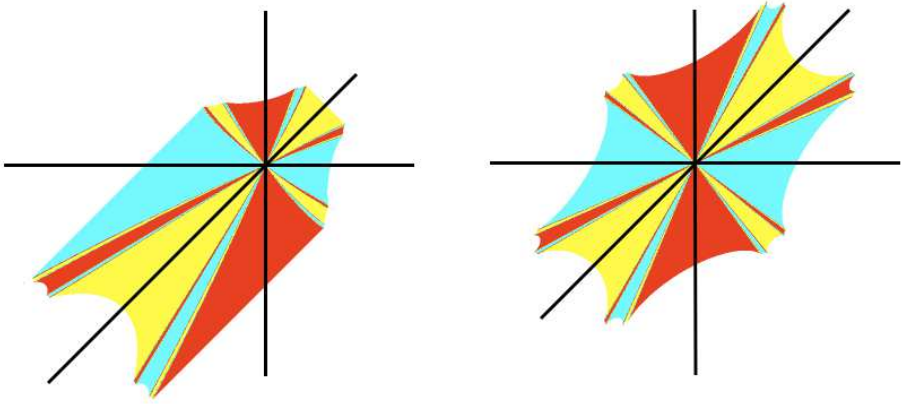}
   \caption{On the left is a real slice of the immediate basin ${\cal B}_{a_\J}(h)$ in the case $|I|=4$ and $\J=\{J_1,J_2\}$ with $|J_1|=3$ and $|J_2|=1$. The critical set $\P(\Delta)$ contains three lines passing through $a_\J$. Almost every orbit in ${\cal B}_{a_\J}(h)$ converges to $a_\J$ tangentially along one of these lines; the immediate basin is colored accordingly. Since $|J_2|=1$, $E_{J_2}=\{0\}$ and $H_{J_1}\oplus H_{J_2}=H_{J_1}$. According to proposition \ref{bottcherprop} below, there is a B\"ottcher coordinate $\Phi: {\cal B}_{a_\J}(h)\to {\cal B}_0(H_{J_1})$. On the right is a real slice of the immediate basin ${\cal B}_0(H_{J_1})$. The isomorphism $\Phi$ respects the coloring.}
   \label{fig:example}
\end{figure}

We will now show that we can apply theorems \ref{theo_local} and \ref{theo_global}. Note that the map $H_{J_1}\oplus H_{J_2}:E_{J_1}\oplus E_{J_2}\to E_{J_1}\oplus E_{J_2}$ is quasihomogeneous of bidegree $(|J_1|,|J_2|)$.

\begin{proposition}\label{bottcherprop}
Let $\J\eqdef\{J_1,J_2\}\in\part$ and $a_\J$ be the image of $L_\J$ in $\P(E)$. Then there is an analytic isomorphism
\[
\Phi:{\cal B}_{a_\J}(h)\rightarrow{\cal B}_0(H_{J_1}\oplus H_{J_2})
\]
conjugating $h$ to $H_{J_1}\oplus H_{J_2}$.
\end{proposition}
\begin{proof}
The analytic map $h:\P(E)\to\P(E)$ is proper, and has a superattracting fixed point at $a_\J$.  We will prove that there is a local isomorphism 
\[
\Phi:\bigl(\P(E),a_\J\bigr)\to\bigl(E_{J_1}\oplus E_{J_2},0\bigr)
\]
conjugating $h$ to $H_{J_1}\oplus H_{J_2}$. Such a $\Phi$ automatically maps the germ of the postcritical set of $h$ at $a_\J$, \ie{} the germ of $\P(\Delta)$ at $a_\J$,  to the germ of the postcritical set of $H_{J_1}\oplus H_{J_2}$ at 0. 
 So near $a_\J$, $\Phi$ maps the postcritical set of $h:{\cal B}_{a_\J}(h)\to{\cal B}_{a_\J}(h)$, that is ${\cal B}_{a_\J}(h)\cap\P(\Delta)$, to the postcritical set of $H_{J_1}\oplus H_{J_2}$. The result then follows from theorem  \ref{theo_global}. 
 
To prove that there is a local conjugacy, we use theorem \ref{theo_local}. Let 
\[E-\{0\}\ni x\mapsto [x]\in \P(E)\] be the natural projection and let $\sigma$ be a local section defined near $a_\J$ in $\P(E)$. The image of $D_{a_\J}\sigma$ is transverse to $L_\J$ which is the kernel of the projection 
\[\pi\eqdef \pi_{J_1}+\pi_{J_2}:E\to E_{J_1}\oplus E_{J_2}.\]
 It follows that  $\pi\circ D_{a_\J}\sigma$ is invertible and so, the composition 
 \[\alpha\eqdef \pi\circ \sigma:(\P(E),a_\J)\to (E_{J_1}\oplus E_{J_2},0)\] 
 is a local isomorphism.
 It conjugates $h:(\P(E),a_\J)\to (\P(E),a_\J)$ to an analytic germ $F:(E_{J_1}\oplus E_{J_2},0)\to (E_{J_1}\oplus E_{J_2},0)$. 
 
Since $H:E\to E$ lifts $h:\P(E)\to \P(E)$, there is an analytic function $\lambda$ defined near $a_\J$ in $\P(E)$ with values in $\C-\{0\}$ such that
 \[\sigma\circ h=\lambda\cdot H\circ \sigma.\]
Set 
\[\mu \eqdef \lambda\circ \alpha^{-1}:(E_{J_1}\oplus E_{J_2},0)\to (\C,\mu_0) \quad\text{with}\quad \mu_0\eqdef \lambda(0).\]
Set $x\eqdef \sigma(0)$ and 
\[\beta\eqdef \sigma\circ \alpha^{-1}:(E_{J_1}\oplus E_{J_2},0)\to (E,x).\]
Then, for $v\in E_{J_1}\oplus E_{J_2}$ sufficiently close to $0$, we have
\[F(v) = \mu(v)\cdot \pi\circ H\bigl(\beta(v)\bigr).\]
In particular, $F=F_{J_1}+F_{J_2}$ with, for $J\in \J$, 
\[F_J(v)=\mu(v)\cdot \pi_{J}\circ H\bigl(\beta(v)\bigr).\]
Note that if $v=v_{J_1}+v_{J_2}$ with $v_J\in E_J$, then $\pi_{J}\bigl(\beta(v)-x\bigr)=v_J$ and according to lemma \ref{lemma_expansion}, for $J\in \J$, we have
\begin{align*}
F_J(v) &= \mu(v)\cdot C_J\cdot H_{J}(v_J)+{\cal O}\bigl(\|\beta(v)-x\|\cdot \|v_J\|^{|J|+1}\bigr)\\
&= \mu_0\cdot C_J\cdot H_{J}(v_J)+{\cal O}\bigl(\|v\|\cdot \|v_J\|^{|J|+1}\bigr).\end{align*}
This shows that $F$ has an adapted superattracting fixed point with quasihomogeneous part $c_1 H_{J_1}\oplus c_2 H_{J_2}$ for some constants $c_1$ and $c_2$ in $\C-\{0\}$. 
   
Assume $v\eqdef v_{J_1}+v_{J_2}$ with $v_J\in E_{J}$ and set $y\eqdef \sigma\circ \alpha^{-1}(v)$. Then $v$ is contained in the postcritical set of $F$ near 0 if and only if $y\in\Delta$. As $v$ tends to $0$, $y$ tends to $x$ and since $x(J_1)\cap x(J_2)=\emptyset$, if $v$ is sufficiently close to 0, the sets $y(J_1)$ and $y(J_2)$ are disjoint. In that case, $y\in \Delta$ if and only if $v_{J_1}:J_1\to\C$ or $v_{J_2}:J_2\to\C$ is not injective. As a consequence, if $v$ is  contained in the postcritical set of $F$ near 0 then $\lambda_1 v_{J_1}+\lambda_2 v_{J_2}$ is contained in the postcritical set of $F$ for all $(\lambda_1,\lambda_2)\in\C^2$ and the vector fields $\theta_{J_1}$ and $\theta_{J_2}$ are tangent to the postcritical set of $F$ near 0. 
 
 According to theorem 1, there is a local isomorphism conjugating $F$ near 0 to $c_1H_{J_1}\oplus c_2 H_{J_2}$ near 0. The existence of a B\"ottcher coordinate 
 \[
 \Phi:(\P(E),a_\J)\to (E_{J_1}\oplus E_{J_2},0)
 \]
 conjugating the map $h:(\P(E),a_\J)\to (\P(E),a_\J)$ to the quasihomogeneous map $H_{J_1}\oplus H_{J_2}:(E_{J_1}\oplus E_{J_2},0)\to (E_{J_1}\oplus E_{J_2},0)$ follows immediately. 
\end{proof}

\section{Questions for further study}\label{questions}

Theorem \ref{theo_global} asserts that when $a_\J\in\P(E)$ is a superattracting fixed point of $h:\P(E)\to\P(E)$ then there is a B\"ottcher coordinate $\Phi:{\cal B}_a(h)\to {\cal B}_0(H_\J)$ with $H_\J\eqdef H_{J_1}\oplus H_{J_2}:E_{J_1}\oplus E_{J_2}\to E_{J_1}\oplus E_{J_2}$. The boundary of ${\cal B}_0(H_\J)$ is a topological sphere of real dimension $2m-5$. Does the inverse $\Phi^{-1}:{\cal B}_0(H_\J)\to {\cal B}_a(h)$ extend continuously to the boundary of ${\cal B}_0(H_\J)$? Is the boundary of ${\cal B}_a(h)$ topologically the quotient of a sphere by an equivalence relation? How could such an equivalence relation be described? 

In dimension one, a global B\"ottcher coordinate gives rise to a dynamical foliation of the immediate basin by rays. What is the higher dimensional analog of these rays? 

In dimension one, if two germs with a superattracting fixed point are topologically conjugate, then this topological conjugacy can be promoted to an analytic conjugacy via a pullback argument. Can one give a topological and/or analytic classification of germs having a superattracting fixed point in higher dimensions? Do these classifications coincide? 

In section \ref{applications}, we applied theorems \ref{theo_local} and \ref{theo_global} to a superattracting fixed point $a_\J$, which was the image of $L_\J$ in $\P(E)$, where $|\J|=2$. What happens for $|\J|>2$? In this case, $\mathrm{dim}(L_\J)=|\J|-1$, and the image in $\P(E)$ will be a projective space of dimension $|\J|-2$. Is this projective space an attractor (in the sense of \cite{mi})? 

If $|\J|>2$ and if $a\in\P(L_\J)$ is a fixed point, then according to proposition \ref{prop_spec} the spectrum of $D_a h:T_a\P(L_\J)\to T_a\P(L_\J)$ belongs to $\C-\overline\D$, and $\P(L_\J)$ is the unstable manifold of $a$. The unique additional eigenvalue of $D_a h:T_a\P(E)\to T_a\P(E)$ is 0. What is the structure of the set 
\[
W^s(a)\eqdef\left\{b\in\P(E):h^{\circ n}(b)\underset{n\to\infty}\longrightarrow a\right\}?
\]
Is it a smooth analytic submanifold of $\P(E)$? Is it dynamically parameterized by the attracting basin of the quasihomogeneous map 
\[
\bigoplus_{J\in\J} H_J:\bigoplus_{J\in\J} E_J\to \bigoplus_{J\in\J} E_J?
\]

The maps to which we applied our theorems in section \ref{applications} were postcritically finite. Are there examples (apart from the quasihomogeneous maps themselves) which are algebraic, and not postcritically finite but admit a B\"ottcher coordinate? The converse is false: consider the map $F:\C^2\to \C^2$ given by $F:(x,y)\mapsto (x^2-y^3,y^2)$. One can verify that $F$ is postcritically finite. The derivative $D_0 F$ is nilpotent so that $F^{\circ 2}$ has a superattracting fixed point at 0. In addition 
\[
F^{\circ 2}(x,y)=(y^2+o(y^2),y^4+o(y^4)).
\]
The map $F^{\circ 2}$ cannot be locally conjugate to the map $(x,y)\mapsto (y^2,y^4)$ as this map is not open.

\end{document}